\documentclass[11pt]{article}
\usepackage{amsfonts}
\usepackage{graphics}
\usepackage{indentfirst}
\usepackage{color}
\usepackage{cite}
\usepackage{latexsym}
\usepackage[paper=a4paper, left=1.6cm, right=1.6cm, top=1.8cm, bottom=1.6cm, headheight=5.5pt, footskip=0.8cm, footnotesep=0.8cm, centering, includefoot]{geometry}
\usepackage{amsmath}
\allowdisplaybreaks
\usepackage{amssymb}
\usepackage[colorlinks, linkcolor=red]{hyperref}
\hypersetup{urlcolor=red, citecolor=red}
\usepackage[dvips]{epsfig}
\usepackage{amscd}

\usepackage{amsthm}
\usepackage{mathrsfs}
\usepackage{verbatim}

\DeclareMathOperator{\divv}{div}
\DeclareMathOperator{\curl}{curl}
\DeclareMathOperator{\loc}{loc}
\DeclareMathOperator{\divf}
{di\overset{\raisebox{0.1ex}{\hspace{0.1em}$\mathbf{\cdot}$}}{v}}
\allowdisplaybreaks

\begin{document}
\title{Global weak solutions and incompressible limit of two-dimensional isentropic compressible magnetohydrodynamic equations with ripped density and large initial data
\thanks{
Wu's research was partially supported by Fujian Alliance of Mathematics (No. 2023SXLMMS08) and the Scientific Research Funds of Xiamen University of Technology (No. YKJ25009R).
Zhong's research was partially supported by Fundamental Research Funds for the Central Universities (No. SWU--KU24001) and National Natural Science Foundation of China (No. 12371227).}
}

\author{Shuai Wang$\,^{\rm 1}\,$,\ Guochun Wu$\,^{\rm 2}\,$,\
Xin Zhong$\,^{\rm 1}\,$ {\thanks{E-mail addresses: swang238@163.com (S. Wang),
guochunwu@126.com (G. Wu), xzhong1014@amss.ac.cn (X. Zhong).}}\date{}\\
\footnotesize $^{\rm 1}\,$
School of Mathematics and Statistics, Southwest University, Chongqing 400715, P. R. China\\
\footnotesize $^{\rm 2}\,$ School of Mathematics and Statistics, Xiamen University of Technology, Xiamen 361024, P. R. China}

\maketitle
\newtheorem{theorem}{Theorem}[section]
\newtheorem{definition}{Definition}[section]
\newtheorem{lemma}{Lemma}[section]
\newtheorem{proposition}{Proposition}[section]
\newtheorem{corollary}{Corollary}[section]
\newtheorem{remark}{Remark}[section]
\renewcommand{\theequation}{\thesection.\arabic{equation}}
\catcode`@=11 \@addtoreset{equation}{section} \catcode`@=12
\maketitle{}

\begin{abstract}
We establish the global existence of weak solutions of the isentropic compressible magnetohydrodynamic equations with ripped density in the whole plane provided the bulk viscosity coefficient is properly large. Moreover, we show that such solutions converge globally in time to a weak solution of the inhomogeneous incompressible magnetohydrodynamic equations when the bulk viscosity coefficient tends to infinity. In particular, {\it the initial energy can be arbitrarily large and vacuum states are allowed in interior regions}. Our analysis depends on the effective viscous flux and a Desjardins-type logarithmic interpolation inequality as well as structure of the
system under consideration. To the best of our knowledge,
this paper provides the first incompressible limit of the isentropic compressible magnetohydrodynamic equations for the large bulk viscosity.
\end{abstract}




\tableofcontents

\section{Introduction}
\subsection{Background and motivation}
The motion of a conducting fluid under the effect of the electromagnetic
field in two dimensions can be described by the following isentropic compressible magnetohydrodynamic (MHD) equations
\begin{align}\label{a1}
\begin{cases}
\rho_t+\divv(\rho\mathbf{u})=0,\\
(\rho\mathbf{u})_t+\divv(\rho\mathbf{u}\otimes\mathbf{u})+\nabla P=
\mu\Delta\mathbf{u}+(\mu+\lambda)\nabla\divv\mathbf{u}+\mathbf{B}\cdot\nabla\mathbf{B}-\frac12\nabla|\mathbf{B}|^2,\\
\mathbf{B}_t+\mathbf{u}\cdot\nabla\mathbf{B}-\mathbf{B}\cdot\nabla\mathbf{u}+\mathbf{B}\divv\mathbf{u}=\nu\Delta\mathbf{B},\\
\divv\mathbf{B}=0,
\end{cases}
\end{align}
where the unknowns $\rho$, $\mathbf{u}=(u^1,u^2)$, $\mathbf{B}=(B^1,B^2)$, and $P=a\rho^\gamma\ (a>0,\gamma>1)$ stand for the density, velocity, magnetic field, and pressure, respectively. The constants $\mu$ and $\lambda$ represent shear viscosity and bulk viscosity, respectively, satisfying the physical restrictions
\begin{equation*}
\mu>0,\ \ \ \mu+\lambda\geq0,
\end{equation*}
while $\nu>0$ is the resistivity coefficient. The system \eqref{a1} couples the isentropic compressible Navier--Stokes equations of fluid dynamics with Maxwell's equations of electromagnetism.

We consider the Cauchy problem of \eqref{a1} with the given initial data
\begin{equation}\label{a2}
(\rho,\mathbf{u},\mathbf{B})|_{t=0}=(\rho_0,\mathbf{u}_0,\mathbf{B}_0)(\mathbf{x}),\ \ \ \mathbf{x}\in\mathbb{R}^2,
\end{equation}
and the far-field behavior
\begin{equation}\label{a3}
(\rho_0,\mathbf{u}_0,\mathbf{B}_0)(\mathbf{x})\rightarrow(\tilde{\rho},\mathbf{0},\mathbf{0}),\ \ \ \text{as}\ |\mathbf{x}|\rightarrow\infty,
\end{equation}
where the constant $\tilde{\rho}>0$ is the fixed reference density.

For $t\geq0$, observe that solutions of \eqref{a1} obey the global energy law
\begin{align*}
\int_{\mathbb{R}^2}\Big(\frac{1}{2}\rho |\mathbf{u}|^2+\frac{1}{2} |\mathbf{B}|^2+G(\rho)\Big)\mathrm{d}\mathbf{x}+\int_0^t\int_{\mathbb{R}^2}\left[\mu|\nabla \mathbf{u}|^2+(\mu+\lambda)(\divv \mathbf{u})^2+\nu|\nabla \mathbf{B}|^2\right]\mathrm{d}\mathbf{x}\mathrm{d}\tau
\leq C_0,
\end{align*}
where the initial total energy is given by
\begin{equation*}
C_0\triangleq\int_{\mathbb{R}^2}
\Big(\frac{1}{2}\rho_0|\mathbf{u}_0|^2+\frac{1}{2}|\mathbf{B}_0|^2+G(\rho_0)\Big)\mathrm{d}\mathbf{x},
\end{equation*}
with the potential energy $G(\rho)$ being defined as
\begin{equation}\label{1.4}
G(\rho)\triangleq\rho\int_{\tilde{\rho}}^\rho\frac{P(v)-P(\tilde{\rho})}{v^2}
\mathrm{d}v.
\end{equation}
Hence, the incompressibility can be recovered formally from the global energy law as $\lambda\rightarrow\infty$.

Magnetohydrodynamics describes the motion of electrically conductive fluids such as plasmas under the influence of electromagnetic fields. It has been found that the MHD equations play a fundamental role in astrophysics, geophysics, plasma physics, and so on (see, e.g., \cite{DA17}). However, the mathematical analysis of the governing system of MHD is exceptionally challenging due to the complex coupling between fluid and magnetic effects. Nevertheless, the pursuit of physical prediction requires a solid mathematical foundation, where a fundamental issue is to establish the global well-posedness of the governing equations---that is, to prove the existence, uniqueness, and continuous dependence on the initial data for all time. The system \eqref{a1} rests on the fundamental constitutive laws, and we refer the reader to \cite[Chapter 3]{LT12} for a rigorous derivation.

For the past four decades, significant progress has been made on the global well-posedness of the solutions for isentropic compressible MHD equations \eqref{a1}. Kawashima \cite{K84} proved the global existence of classical solutions for the 2D Cauchy problem provided the initial data are close to the constant state. Later on, based on the classical framework of Matsumura and Nishida \cite{MN80,MN83} for compressible viscous flows, Li--Yu \cite{LY11} and Zhang--Zhao \cite{ZZ10} independently established the existence and uniqueness of the global classical solutions for initial data that are small perturbations in $H^3(\mathbb{R}^3)$ surrounding a non-vacuum equilibrium state. Beyond this setting, as an equally mathematically challenging yet physically significant case, the possible presence of vacuum presents difficulties within the small-perturbation framework where the governing equations become singular and degenerate near the vacuum region. In this regards, it has been shown in \cite{LXZ13,LSX16} that the classical solutions to the 3D and 2D Cauchy problem exist globally, with initial data having small energy but possibly large oscillations and vacuum states, respectively.

From the viewpoint of partial differential equations, it is certainly interesting to weaken or remove the small-energy assumption. The result in \cite{HHPZ17} reveals that this relaxation is possible, as the authors proved the global well-posedness of classical solutions for \eqref{a1} in $\mathbb{R}^3$ provided $[(\gamma-1)^{1/9}+\nu^{-1/4}]C_0$ is suitably small. Moreover, with the aid of weak convergence method developed by P.-L. Lions \cite{PL98} and Feireisl \cite{F04}, Hu and Wang \cite{HW10} demonstrated the global existence and large-time behavior of finite-energy large weak solutions in 3D bounded domains with Dirichlet boundary conditions, when the adiabatic exponent $\gamma>\frac32$. These weak solutions provide valuable insights into the system's behavior in extreme regimes although they lack uniqueness and perfect regularities.

The pursuit of a well-posedness theory, particularly the issue of uniqueness, naturally leads to the search for weak solutions with enhanced regularity.
Beyond the well-known {\it finite-energy weak solutions} \cite{HW10} and {\it small-energy classical solutions} \cite{LY11,LXZ13,ZZ10,LSX16}, a third important category is that of the so-called {\it Hoff's intermediate weak solutions} \cite{Hoff95,Hoff95*}.
These solutions possess stronger regularity than finite-energy weak solutions in the sense that particle paths can be defined in non-vacuum regions, yet weaker than standard strong solutions as they may have discontinuous density along some curves (in 2D) or surfaces (in 3D) (see \cite{Hoff02}). Suen and Hoff \cite{SH12} proved the global existence of such solutions to \eqref{a1} for initial data that are small in $L^2(\mathbb{R}^3)$ and have positive essentially bounded initial density. Such a momentous result was later generalized in \cite{LYZ13} to allow for initial vacuum.
This improved regularity subsequently enabled the establishment of uniqueness and continuous dependence for these weak solutions in \cite{S20}, thereby strengthening the existence theory in \cite{SH12}. Outside of these small-data settings, an aspect worth examining is whether Hoff's intermediate weak solutions can be constructed for large initial data with vacuum.

Recently, when the bulk viscosity is sufficiently large, Danchin and Mucha \cite{DM17} obtained global regular solutions to the isentropic compressible Navier--Stokes equations (i.e., \eqref{a1} with $\mathbf{B}\equiv\mathbf{0}$) in $\mathbb{R}^2$ with initial density bounded away from vacuum and arbitrarily large initial velocity. As a by-product, they derived the incompressible limit as $\lambda\rightarrow\infty$ with a convergence rate of order $(2\mu+\lambda)^{-1/2}$. Subsequently, Danchin and Mucha \cite{DM23} generalized the result in \cite{DM17} to the torus $\mathbb{T}^2$,
allowing the presence of vacuum but assuming zero initial total momentum. One may ask: can the initial non-vacuum condition in \cite{DM17} and the zero initial momentum assumption in \cite{DM23} be removed? Furthermore, one would like to wonder whether it is possible to establish analogous results to the two-dimensional isentropic compressible MHD system \eqref{a1}, investigating what role the magnetic field plays.

The primary objective of the current paper is to provide an affirmative answer to these issues. More precisely, we shall address the Cauchy problem \eqref{a1}--\eqref{a3} supplemented with general arbitrary large initial data with merely nonnegative bounded density: we recall that a ``ripped'' initial density is a function that may have nontrivial regions of vacuum, without any extra regularity assumption (see \cite{DM23}).
We will prove global weak solutions to the Cauchy problem \eqref{a1}--\eqref{a3} with large initial data as long as the bulk viscosity coefficient is properly large and investigate the limiting behavior of such solutions as $\lambda\rightarrow\infty$. One expects to obtain in the limit a global solution to the inhomogeneous incompressible magnetohydrodynamic equations
\begin{align}\label{1.5}
\begin{cases}
\varrho_t+{\bf v}\cdot\nabla\varrho=0,\\
\varrho{\bf v}_t+\varrho{\bf v}\cdot\nabla {\bf v}+\nabla \Pi=\mu\Delta {\bf v}+{\bf b}\cdot\nabla{\bf b}-\frac12\nabla|{\bf b}|^2,\\
{\bf b}_t+\mathbf{v}\cdot\nabla{\bf b}={\bf b}\cdot\nabla\mathbf{v}+\nu\Delta{\bf b},\\
\divv{\bf v}=\divv{\bf b}=0,\\
(\varrho,{\bf v},{\bf b})|_{t=0}=(\rho_0,{\bf v}_0,{\bf B}_0)
\end{cases}
\end{align}
for $({\bf x},t)\in\mathbb{R}^{2}\times(0,+\infty)$, where ${\bf v}_0$ is the Leray-Helmholtz projection of ${\bf u}_0$ on divergence-free vector fields.

\subsection{Main results}
Before presenting our main results, we first formulate the notations and conventions adopted in this paper. We denote by $C$ a generic positive constant which may vary at different places. The symbol $\Box$ marks the end of a proof, $A:B$ represents the trace of the matrix product $AB^\top$, and $c\triangleq d$ means $c=d$ by definition. For notational simplicity, we write
\begin{align*}
\int f \mathrm{d}\mathbf{x}=\int_{\mathbb{R}^2} f \mathrm{d}\mathbf{x}, \ \ f_i=\partial_if\triangleq\frac{\partial f}{\partial x_i}.
\end{align*}
For $1\le p\le \infty$ and integer $k\ge 0$, we denote the standard Sobolev spaces
\begin{align*}
L^p=L^p(\mathbb{R}^2),\ \ W^{k, p}=W^{k, p}(\mathbb{R}^2),\ \ H^k=W^{k, 2},\ \
D^{1,p}=\{f\in L_{\loc}^1(\mathbb{R}^2):\|\nabla f\|_{L^p}<\infty\}.
\end{align*}
Furthermore, for any $f\in L^1_{\loc}(\mathbb R^2)$, we define its mollification by $[f]_\epsilon\triangleq j_\epsilon*f$, where $j_\epsilon=j_\epsilon({\bf x})$ is the stand mollifier with width $\epsilon$. Whenever $\alpha\in (0,1]$,
the H\" older seminorm of a function ${\bf v}:\mathbb R^2\rightarrow \mathbb R^2$ is defined by
\begin{align*}
\langle {\bf v}\rangle^\alpha
=\sup\limits_{\substack{{\bf x}, {\bf y}\in\mathbb R^2\\ {\bf x}\neq {\bf y}}}
\frac{|{\bf v}({\bf x})-{\bf v}({\bf y})|}{|{\bf x}-{\bf y}|^\alpha}.
\end{align*}

In addition, we denote by
\begin{align}\label{1.6}
\begin{cases}
\dot{f}\triangleq f_t+\mathbf{u}\cdot\nabla f,\\
F\triangleq(2\mu+\lambda)\divv\mathbf{u}-(P(\rho)-P(\tilde{\rho}))-\frac12|\mathbf{B}|^2,\\
\omega\triangleq\curl\mathbf{u}=-\nabla^\bot\cdot\mathbf{u}=-\partial_2u^1+\partial_1u^2,
\end{cases}
\end{align}
which represent the material derivative of $f$, the effective viscous flux, and the vorticity, respectively. Then it follows from $\eqref{a1}_2$ that
\begin{equation}\label{1.7}
  \Delta F=\divv(\rho\dot{\mathbf{u}}-\mathbf{B}\cdot\nabla\mathbf{B}), \ \ \ \mu\Delta\omega=\curl(\rho\dot{\mathbf{u}}-\mathbf{B}\cdot\nabla\mathbf{B}).
\end{equation}
Moreover, we introduce the Leray-Helmholtz projection
\begin{equation*}
\mathcal{P}\triangleq \text{Id}+\nabla(-\Delta)^{-1}\divv
\end{equation*}
onto the subspace of divergence-free vector fields and $\mathcal{Q}\triangleq \text{Id}-\mathcal{P}$. In particular, because
$\mathcal P$ and $\mathcal Q$ are smooth homogeneous of degree $0$ Fourier multipliers, they map $L^p$ into itself for any $1<p<\infty$.

We recall the definition of weak solutions to the Cauchy problem \eqref{a1}--\eqref{a3} in the sense of \cite{Hoff95,Hoff95*}.
\begin{definition}\label{d1.1}
A triplet $(\rho, \mathbf{u}, \mathbf{B})$ is said to be a weak solution to the problem \eqref{a1}--\eqref{a3} provided that
\begin{equation*}
(\rho-\tilde{\rho}, \rho\mathbf{u}, \mathbf{B}) \in C([0,\infty);H^{-1}(\mathbb{R}^2)),\ \
(\nabla\mathbf{u}, \nabla\mathbf{B})\in L^2(\mathbb{R}^2\times(0,\infty))
\end{equation*}
with $\divv \mathbf{B}(\cdot,t)=0$ in $\mathcal{D}'(\mathbb{R}^2)$ for $t>0$ and  $(\rho,\mathbf{u},\mathbf{B})|_{t=0}=(\rho_0,\mathbf{u}_0,\mathbf{B}_0)$.
Moreover, for any $t_2\ge t_1\ge 0$ and any $C^1$ test function $\psi({\bf x},t)$ with uniformly bounded support in ${\bf x}$ for $t\in[t_1,t_2]$, the following identities hold\footnote{Throughout this paper, we will use the Einstein summation over repeated indices convention.}:

\begin{align*}
&\int\rho(\mathbf{x},\cdot)\psi(\mathbf{x},\cdot)
\mathrm{d}\mathbf{x}\Big|_{t_1}^{t_2}=\int_{t_1}^{t_2}\int(\rho\psi_t+
\rho\mathbf{u}\cdot\nabla\psi)\mathrm{d}\mathbf{x}\mathrm{d}t,\\
&\int(\rho u^j)(\mathbf{x},\cdot)\psi
(\mathbf{x},\cdot)\mathrm{d}\mathbf{x}\Big|_{t_1}^{t_2}
+\int_{t_1}^{t_2}
\int\big(\mu\nabla u^j\cdot\nabla\psi+(\mu+\lambda)\divv\mathbf{u}\psi_j\big)
\mathrm{d}\mathbf{x}\mathrm{d}t\notag\\
&\qquad=\int_{t_1}^{t_2}
\int\Big(\rho u^j\psi_t+\rho u^j\mathbf{u}\cdot\nabla\psi+P\psi_j
+\frac{1}{2}|\mathbf{B}|^2\psi_j-B^j\mathbf{B}\cdot\nabla\psi\Big)
\mathrm{d}\mathbf{x}\mathrm{d}t,\\
&\int B^j(\mathbf{x},\cdot)\psi(\mathbf{x},\cdot)
\mathrm{d}\mathbf{x}\Big|_{t_1}^{t_2}=\int_{t_1}^{t_2}\int\big(B^j\psi_t+
B^j\mathbf{u}\cdot\nabla\psi-u^j\mathbf{B}\cdot\nabla\psi-\nu\nabla B^j\cdot\nabla\psi\big)\mathrm{d}\mathbf{x}\mathrm{d}t.
\end{align*}
\end{definition}
Concerning the initial data $(\rho_0,\mathbf{u}_0,\mathbf{B}_0)$, we assume that there exist two positive constants $\hat{\rho}$ and $M$ (not necessarily small) satisfying
\begin{gather}\label{c1}
0\leq\inf\rho_0\leq\sup\rho_0\leq\hat{\rho},\ \ \divv \mathbf{B}_0=0,
\\
C_0+\mu\|\nabla\mathbf{u}_0\|_{L^2}^2+(\mu+\lambda)\|\divv\mathbf{u}_0\|_{L^2}^2+\nu\|\nabla\mathbf{B}_0\|_{L^2}^2
\leq M.\label{c2}
\end{gather}
Additionally, \eqref{1.4} implies that for some $C=C(\tilde{\rho},\hat{\rho})>0$,
\begin{equation}\label{1.10}
  \frac{1}{C}(\rho-\tilde{\rho})^2\leq G(\rho)\leq C(\rho-\tilde{\rho})^2.
\end{equation}
Here and hereafter, we occasionally write $C(f)$ to denote the dependence of a generic constant on $f$.

We now state our first result on the global existence of weak solutions.

\begin{theorem}\label{t1.1}
Let \eqref{c1} and \eqref{c2} be satisfied, there exists a positive number $D$ depending only on $\tilde{\rho}$, $\hat{\rho}$, $a$, $\gamma$, $\nu$, and $\mu$ such that if
\begin{equation}\label{lam}
\lambda\geq\exp\bigg\{(2+M)^{\exp\big\{D(1+C_0)^6\big\}}\bigg\},
\end{equation}
then the Cauchy problem \eqref{a1}--\eqref{a3} admits a global weak solution $(\rho,\mathbf{u},\mathbf{B})$ in the sense of Definition $\ref{d1.1}$ satisfying
\begin{equation}\label{reg}
\begin{cases}0\leq\rho(\mathbf{x},t)\leq2\hat{\rho}~a.e.~\mathrm{on}~\mathbb{R}^2\times[0,\infty),\\
(\rho-\tilde{\rho},\sqrt{\rho}\mathbf{u},\mathbf{B})\in C([0,\infty);L^2(\mathbb{R}^2)),~(\mathbf{u},\mathbf{B})\in L^\infty(0,\infty;H^1(\mathbb{R}^2)),\\
(\nabla^2\mathcal{P}\mathbf{u},\nabla F,\sqrt{\rho}\dot{\mathbf{u}},\mathbf{B}_t,\nabla^2\mathbf{B})\in L^2(\mathbb{R}^2\times(0,\infty)),\\
\sigma^{\frac{1}{2}}(\sqrt{\rho}\dot{\mathbf{u}},\mathbf{B}_t,\nabla^2\mathbf{B})\in L^\infty(0,\infty;L^2(\mathbb{R}^2)),~ \sigma^{\frac{1}{2}}(\nabla\dot{\mathbf{u}},\nabla\mathbf{B}_t)
\in L^2(\mathbb{R}^2\times(0,\infty)),
\end{cases}
\end{equation}
where $\sigma\triangleq\min\{1,t\}$.
\end{theorem}

The next result will treat the incompressible limit (characterised by the large value of the bulk viscosity) of the global weak solutions established in Theorem \ref{t1.1}.

\begin{theorem}\label{t1.2}
Let $\{(\rho^\lambda,{\bf u}^\lambda,{\bf B}^\lambda)({\bf x},t)\}$ be the family of solutions from Theorem \ref{t1.1}. Then, there exists a subsequence $\{\lambda_k\}$
with $\lambda_k\rightarrow\infty$ satisfying
\begin{align}\label{1.13}
 \rho^{\lambda_{k}}\rightarrow \varrho ~~&\text{strongly in} ~ L^2(\mathbb R^2)\ \text{for any}\ t\ge 0,\\
 {\bf u}^{\lambda_{k}}\rightarrow {\bf v},~~{\bf B}^{\lambda_{k}}\rightarrow {\bf b}
 ~~&\text{uniformly on compact sets in}~\mathbb R^2\times(0,\infty),\notag
\end{align}
where $(\varrho,{\bf v},{\bf b})$ is a global weak solution to the inhomogeneous incompressible magnetohydrodynamic equations \eqref{1.5} in the sense of Definition \ref{d1.2} below.
\end{theorem}

\begin{definition}\label{d1.2}
A triplet $(\varrho, {\bf v}, {\bf b})$ is said to be a weak solution to Cauchy problem \eqref{1.5} provided that
\begin{gather}
\varrho\in L^\infty(\mathbb R^2\times (0,\infty)),~~
(\sqrt{\varrho}\mathbf{v},\mathbf{b})\in L^\infty([0,\infty); L^2(\mathbb{R}^2)),
\notag\\
(\nabla{\bf v},\nabla\mathbf{b})\in L^2(\mathbb R^2\times (0,\infty)),\notag\\
(\varrho-\tilde{\rho})\in C([0,\infty);L^{2}(\mathbb R^2)),\label{1.14}\\
 {\bf v}\in L^2(\mathbb R^2\times (0,T)), ~~ \text{for any} \ T>0.\label{1.15}
\end{gather}
Moreover, for any $t_2\geq t_1\geq 0$ and any $C^1$ test function $(\phi,\boldsymbol\psi)$,
with uniformly bounded support in ${\bf x}$ for $t\in[t_1,t_2]$ and satisfying $\divv\boldsymbol\psi(\cdot,t)=0$ on $\mathbb R^2\times[0,\infty)$, the following identities hold:
\begin{align}\label{1.16}
\int\varrho({\bf x},\cdot)\phi({\bf x},\cdot)\mathrm{d}{\bf x}\Big |_{t_1}^{t_2}&=\int_{t_1}^{t_2}\int(\varrho\phi_t+\varrho{\bf v}\cdot\nabla\phi)\mathrm{d}{\bf x}\mathrm{d}t,
\\
\int(\varrho{\bf v}\cdot\boldsymbol\psi)({\bf x},\cdot)\mathrm{d}{\bf x}\Big |_{t_1}^{t_2}
&=
\int_{t_1}^{t_2}\int\big(\varrho{\bf v}\cdot\boldsymbol\psi_t+\varrho{ v}^i{\bf v}\cdot\nabla_i\boldsymbol\psi-\mu\nabla{\bf v}:\nabla\boldsymbol\psi-b^j\mathbf{b}\cdot\nabla\psi^j\big)\mathrm{d}{\bf x}\mathrm{d}t,\label{1.17}
\\
\int({\bf b}\cdot\boldsymbol\psi)({\bf x},\cdot)
\mathrm{d}\mathbf{x}\Big|_{t_1}^{t_2}&=\int_{t_1}^{t_2}\int
\big({\bf b}\cdot\boldsymbol\psi_t+b^j\mathbf{v}\cdot\nabla\psi^j-v^j\mathbf{b}\cdot\nabla\psi^j
-\nu\nabla{\bf b}:\nabla\boldsymbol\psi\big)
\mathrm{d}\mathbf{x}\mathrm{d}t.\label{1.18}
\end{align}
\end{definition}

Several remarks are in order.

\begin{remark}
It should be noted that Theorem \ref{t1.1} holds for arbitrarily large initial energy as long as the bulk viscosity coefficient is suitably large, which is in sharp contrast to \cite{LXZ13,SH12} where the small initial energy is needed.
\end{remark}

\begin{remark}
It should be pointed out that Theorems \ref{t1.1} and \ref{t1.2} also hold for the isentropic compressible Navier--Stokes equations (i.e., $\mathbf{B}\equiv\mathbf{B}_0\equiv\mathbf{0}$). Thus, we extend the significant results obtained by Danchin and Mucha in \cite{DM17,DM23} to the case of isentropic compressible MHD equations. However, this is a non-trivial generalization. On the one hand, compared with \cite[Theorem 1.1 and Corollary 1.1]{DM17}, the density in our theorems is allowed to have large oscillation and vacuum states in interior regions. On the other hand, we do not require the assumption of zero initial total momentum in comparison with \cite[Theorems 2.1 and 2.3]{DM23}.
Furthermore, it seems that the methods used in \cite{DM23} could not be employed to tackle the case of unbounded domains directly.
\end{remark}

\begin{remark}
We emphasize that non-vacuum behavior at infinity for the density plays a crucial role in our analysis as we can use Desjardins-type logarithmic interpolation inequality (Lemma 2.4 in \cite{Z22}) in this case to obtain the high integrability of velocity. It is natural to ask whether analogous results hold for the 2D Cauchy problem with
far-field vacuum. New ideas are required to handle this case, which will be left for future studies.
\end{remark}

\subsection{Strategy of the proof}

We now outline the main challenges and strategies. To prove Theorem \ref{t1.1}, we construct global smooth approximate solutions by applying the local existence theory with strictly positive initial density and the blow-up criterion (see Lemma \ref{l2.1}). The proof of Theorem \ref{t1.2} then follows from the uniform estimates in Theorem \ref{t1.1} via compactness arguments. Therefore, the core of our analysis is to establish uniform {\it a priori} estimates independent of the lower bound of density $\rho$ and the bulk viscosity $\lambda$. It should be pointed out that the crucial techniques used in \cite{DM23} dealing with isentropic compressible Navier--Stokes equations cannot be adopted to the situation treated here, since their arguments depend heavily on the boundedness of domains and the assumption of zero initial total momentum. Moreover, the presence of initial vacuum and large oscillation leads to more complicated analysis than the previous 2D Cauchy problem \cite{DM17}. Furthermore, the absence of small initial energy (when $\lambda$ is large enough) will bring out some new difficulties in comparison with the 3D Cauchy problem \cite{HWZ}. Last but not least, the strong
velocity-magnetic field coupling and nonlinear terms like ${\bf u}\cdot\nabla {\bf B}$
and ${\bf B}\cdot\nabla {\bf B}$ pose additional challenges. Consequently, some new ideas are needed to overcome these obstacles.

It was shown (with minor modification) in \cite{WT} that if $0<T^*<\infty$ is the maximal existence time of strong solutions to \eqref{a1}--\eqref{a3}, then
\begin{align*}
\limsup\limits_{T\nearrow T^*}\|\rho\|_{L^\infty(0,T;L^\infty)}=\infty,
\end{align*}
which implies that the key uniform-in-$\lambda$ {\it a priori} estimate is to obtain the time-independent upper bound of the density.
First, we attempt to obtain the estimate on the $L^\infty(0, T; L^2)$-norm of the gradient of velocity (see Lemma \ref{l3.2}). To do so, multiplying \eqref{a1}$_2$ by ${\bf u}_t$, we find that one of the key points is to control $\|{\bf u}\|_{L^4}$, which in turn depends on the bounds of $\|{\bf u}\|_{H^1}$. Due to the absence of vacuum at infinity, we observe that the $L^2$-norm of ${\bf u}$ can be bounded by $C_0$ and $\|\nabla{\bf u}\|_{L^2}$ with the help of H{\"o}lder's inequality and Gagliardo--Nirenberg inequality (see \eqref{3.12}).
In addition, from the elementary energy estimate, a key observation is that the periods of large velocity gradient $\|\nabla\mathbf{u}\|_{L^2}$
can only be sustained for short (i.e., measure-limited) intervals of time.
Based on these, we thus divide the arguments concerning the second term $\int\rho \dot{\mathbf{u}}\cdot(\mathbf{u}\cdot\nabla)\mathbf{u}\mathrm{d}\mathbf{x}$ on the right-hand side of \eqref{3.10} into two cases: $\|\nabla\mathbf{u}\|_{L^2}\leq1$ and $\|\nabla\mathbf{u}\|_{L^2}\geq1$, which in turn need to derive the bound of $\|\sqrt{\rho}{\bf u}\|_{L^4}$.
By a crucial Desjardins-type logarithmic interpolation inequality (see Lemma $\ref{log}$), this can be controlled through $\frac{1}{(2\mu+\lambda)^2}\|P(\rho)-P(\tilde{\rho})\|_{L^4}^4$ (see \eqref{3.14}--\eqref{3.16}).

Hence, the crucial issue is the integrability in time of $\frac{1}{(2\mu+\lambda)^2}\|P(\rho)-P(\tilde\rho)\|_{L^4}^4$ (see \eqref{3.26}). To overcome this difficulty, we assume the {\it a priori hypothesis} \eqref{3.1} in Proposition \ref{p3.1}. So, the next key step is to complete the proof of the {\it a priori hypothesis}, that is, to show \eqref{3.2}. We observe that the \textit{effective viscous flux} generates a damping term which provides strong density dissipation, thereby motivating us to build the bound for $L^1(0,T)$-norm of $\frac{1}{2\mu+\lambda}\|P(\rho)-P(\tilde{\rho})\|_{L^4}^4$ (see Lemma $\ref{l3.3}$). At the same time, the appearance of a magnetic field requires the bounds of $\|{\bf B}\|_{L^4}$. To these ends, from the momentum equation \eqref{a1}$_2$, we note that
\begin{equation*}
  \divv\mathbf{u}=\frac{-(-\Delta)^{-1}\divv(\rho\dot{\mathbf{u}}-\mathbf{B}\cdot\nabla\mathbf{B})+P(\rho)-P(\tilde{\rho})+\frac12|\mathbf{B}|^2}{2\mu+\lambda},
\end{equation*}
which further inspires us to establish the estimates on the material derivative of the velocity (see Lemmas $\ref{l3.2}$ and $\ref{l3.4}$).

However, since the bulk viscosity is often coupled with the divergence of the velocity while being inherently contained in the \textit{effective viscous flux}, it seems hard to derive the $L^\infty(0,\min\{1,T\};L^2)$-norm for the material derivative of the velocity (see Lemma $\ref{l3.4}$) in order to isolate $\lambda$ (see, e.g., \eqref{3.39}--\eqref{3.41}). On the one hand, it is challenging to obtain the estimates for the term
\begin{equation*}
  (\mu+\lambda)\sigma\int\dot{u}^j\left(\partial_j\divv\mathbf{u}_t +\divv(\mathbf{u}\partial_j\divv\mathbf{u})\right)\mathrm{d}\mathbf{x}.
\end{equation*}
Thus, motivated by \cite{HWZ}, we shall consider (see \eqref{3.36})
\begin{equation*}
  (\mu+\lambda)\int(\divf\mathbf{u})^2\mathrm{d}\mathbf{x}~~~~
\text{rather than}
~~~~(\mu+\lambda)\int(\divv\dot{\mathbf{u}})^2\mathrm{d}\mathbf{x},
\end{equation*}
where
\begin{equation*}
\divf\mathbf{u}\triangleq\divv\mathbf{u}_t+\mathbf{u}\cdot\nabla\divv\mathbf{u}
=\divv\dot{\mathbf{u}}-u_j^iu_i^j.
\end{equation*}
On the other hand, we have to use the Hodge-type decomposition to isolate the ``bad" terms (divergence-free part) from $\nabla\mathbf{u}$, producing additional troublesome terms,
which are overcome by using integration by parts (see, e.g., \eqref{3.41}). Then we succeed in deriving the desired estimates on $L^\infty(0,\min\{1,T\};L^2)$-norm (see Lemma $\ref{l3.4}$).
Having these estimates at hand, we can
obtain the time-independent upper bound of the density by applying Lagrangian
coordinates technique used in \cite{DE97} (see Lemma $\ref{l3.5}$).
Finally, the proof of \textit{a priori hypothesis} is complete once the bulk viscosity is properly large (see {\it Proof of Proposition \ref{p3.1}}).
Let us emphasize that the effective viscous flux and Desjardins-type logarithmic interpolation inequality play essential roles in our analysis.

The rest of the paper is organized as follows. In the next section, we recall some known facts and elementary inequalities that will be used later. Section \ref{sec3} is devoted to obtaining {\it a priori} estimates. Then we give the proof of Theorem \ref{t1.1} in Section \ref{sec4}, while the proof of Theorem \ref{t1.2} is carried over to the last section.

\section{Preliminaries}\label{sec2}

This section compiles several well-known facts and elementary inequalities.

\subsection{Auxiliary inequalities and facts}
In this subsection we review some known facts and inequalities.
We begin with a lemma concerning the local existence and the possible breakdown of strong solutions to the problem \eqref{a1}--\eqref{a3}, which has been proven in \cite{WT}.
\begin{lemma}\label{l2.1}
Assume that
\begin{equation*}
(\rho_0-\tilde{\rho},\mathbf{u}_0,\mathbf{B}_0)\in H^{2},~~\divv \mathbf{B}_0=0, ~~\text{and}~~\inf\limits_{\mathbf{x}\in\mathbb{R}^2}\rho_0(\mathbf{x})>0,
\end{equation*}
then there exists a positive constant $T$ such that the Cauchy problem \eqref{a1}--\eqref{a3} admits a unique strong solution $(\rho,\mathbf{u},\mathbf{B})$ satisfying
\begin{equation*}
(\rho-\tilde{\rho},\mathbf{u},\mathbf{B})\in C([0,T]; H^{2}),~~\text{and}~~ \inf_{\mathbb{R}^2\times[0,T]}\rho(\mathbf{x},t)\geq\frac{1}{2}
\inf_{\mathbf{x}\in\mathbb{R}^2}\rho_0(\mathbf{x})>0.
\end{equation*}
Moreover, if $T^*$ is the maximal time of existence, then it holds that
\begin{equation*}
\lim\sup_{T\nearrow T^*}\|\rho\|_{L^\infty(0,T;L^\infty)}=\infty.
\end{equation*}
\end{lemma}

The following well-known Gagliardo--Nirenberg inequality (see \cite[Theorem 12.83]{LG}) will be employed.
\begin{lemma}\label{GN}
Let $p\in [2, \infty)$, $q\in(1, \infty)$, and $r\in (2, \infty)$. Then there exists some generic constant $C>0$ which may depend on $p$, $q$, and $r$ such that, for $f\in H^1$ and $g\in L^q\cap D^{1, r}$,
\begin{gather*}
\|f\|_{L^p}\leq C\|f\|_{L^2}^\frac{2}{p}\|\nabla f\|_{L^2}^{1-\frac{2}{p}},\ \ \
\|g\|_{L^\infty}\leq C\|g\|_{L^q}^\frac{q(r-2)}{2r+q(r-2)}\|\nabla g\|_{L^r}^\frac{2r}{2r+q(r-2)}.
\end{gather*}
\end{lemma}

Next, we have the following Desjardins-type logarithmic interpolation inequality (see \cite[Lemma 2.4]{Z22}), which extends the case of two-dimensional torus $\mathbb{T}^2$ treated in \cite[Lemma 2]{DE97} (see also \cite[Lemma 1]{D1997}) to the whole plane $\mathbb{R}^2$.
\begin{lemma}\label{log}
Assume that $0\le\rho\le\hat{\rho}$ and $\mathbf{u}\in H^1(\mathbb{R}^2)$, then
\begin{equation*}
\|\sqrt\rho\mathbf{u}\|_{L^4}^2\leq C(\hat{\rho})(1+\|\sqrt\rho\mathbf{u}\|_{L^2})\|\mathbf{u}\|_{H^1}
\sqrt{\ln\left(2+\|\mathbf{u}\|_{H^1}^2\right)}.
\end{equation*}
\end{lemma}

Finally, we recall the following commutator estimates in \cite[Lemma 2.3]{PL96}, which play an important role in the mollifier arguments.
\begin{lemma}\label{lcom}
Let $\rho \in L^{p}(\mathbb R^2)$ and ${\bf u} \in W^{1,q}(\mathbb R^2)$ be given functions with $1\le p,q<\infty$ and $\frac{1}{p}+\frac{1}{q}\le 1$. For any $\epsilon>0$, then we have
\begin{equation*}
\|\divv[\rho{\bf u}]_\epsilon-\divv\left([\rho]_\epsilon{\bf u}\right)\|_{L^r}\le C\|\rho\|_{L^p}\|{\bf u}\|_{W^{1,q}},
\end{equation*}
and
\begin{equation*}
\divv[\rho{\bf u}]_\epsilon-\divv\left([\rho]_\epsilon{\bf u}\right)\rightarrow 0\ \text{ in}\ \ L^r(\mathbb R^2)\ \ \text{as}\ \epsilon\rightarrow0,
\end{equation*}
where $C\geq0$ is independent of $\epsilon,\rho$, and ${\bf u}$, and $r$ is determined by $\frac{1}{r}=\frac{1}{p}+\frac{1}{q}$.
\end{lemma}

\subsection{Uniform estimates for $F$, $\boldsymbol\omega$, and $\nabla\mathbf{u}$}

The following estimates follow from \eqref{1.7}, the fact $2\mu+\lambda\geq\mu>0$, Gagliardo--Nirenberg inequality, and the standard $L^p$-estimate for elliptic systems.
\begin{lemma}\label{E0}
Let $(\rho,\mathbf{u},\mathbf{B})$ be a smooth solution to the problem \eqref{a1}--\eqref{a3}. Then, for any $2\leq p<\infty$, there exists a generic positive constant $C$ depending only on $p$ and $\mu$ such that
\begin{equation}\label{E1}
\|\nabla \mathbf{u}\|_{L^p}\leq C\big(\|\divv \mathbf{u}\|_{L^p}+\|\omega\|_{L^p}\big),
\end{equation}
\begin{equation}\label{E2}
\|\nabla F\|_{L^p}+\|\nabla^2\mathcal{P}\mathbf{u}\|_{L^p}+\|\nabla\omega\|_{L^p}\leq
C\big(\|\rho\dot{\mathbf{u}}\|_{L^p}+\|\mathbf{B}\cdot\nabla\mathbf{B}\|_{L^p}\big),
\end{equation}
\begin{equation}\label{E3}
\|\nabla\mathcal{P}\mathbf{u}\|_{L^p}+
\|\omega\|_{L^p}\leq C\big(\|\rho\dot{\mathbf{u}}\|_{L^2}+\|\mathbf{B}\cdot\nabla\mathbf{B}\|_{L^2}\big)^{1-\frac{2}{p}}\|\nabla \mathbf{u}\|_{L^2}^\frac{2}{p},
\end{equation}
\begin{equation}\label{E4}
\|F\|_{L^p}\leq
C\big(\|\rho\dot{\mathbf{u}}\|_{L^2}
+\|\mathbf{B}\cdot\nabla\mathbf{B}\|_{L^2}\big)^{1-\frac{2}{p}}
\big[(2\mu+\lambda)\|\divv\mathbf{u}\|_{L^2}+
\|P-P(\tilde{\rho})\|_{L^2}+\|\mathbf{B}\|_{L^4}^2\big]^\frac{2}{p},
\end{equation}
\begin{equation}\label{E5}
\begin{aligned}[b]
\|\nabla \mathbf{u}\|_{L^p}
&\leq
C\big(\|\rho\dot{\mathbf{u}}\|_{L^2}+\|\mathbf{B}\cdot\nabla\mathbf{B}\|_{L^2}\big)^{1-\frac{2}{p}}\|\nabla \mathbf{u}\|_{L^2}^\frac{2}{p}
+\frac{C}{2\mu+\lambda}\|P-P(\tilde{\rho})\|_{L^p}
+\frac{C}{2\mu+\lambda}\|\mathbf{B}\|_{L^{2p}}^2\\
&\quad+\frac{C}{2\mu+\lambda}\big(\|\rho\dot{\mathbf{u}}\|_{L^2}
+\|\mathbf{B}\cdot\nabla\mathbf{B}\|_{L^2}\big)^{1-\frac{2}{p}}
\big(\|P-P(\tilde{\rho})\|_{L^2}+\|\mathbf{B}\|_{L^4}^2\big)^\frac{2}{p}.
\end{aligned}
\end{equation}
\end{lemma}

\section{\textit{A priori} estimates}\label{sec3}
In this section we will show some necessary {\it a priori} bounds for the strong solutions guaranteed by Lemma $\ref{l2.1}$. It should be stressed that these bounds are independent of the bulk viscosity $\lambda$, the lower bound of $\rho$, the initial regularity, and the time of existence. More precisely, let $T>0$ be fixed and $(\rho, \mathbf{u}, \mathbf{B})$ be the strong solution to \eqref{a1}--\eqref{a3} in $\mathbb{R}^2\times(0, T]$, we can obtain the following key {\it a priori} estimates on $(\rho, \mathbf{u}, \mathbf{B})$.
\begin{proposition}\label{p3.1}
Under the conditions of Theorem $\ref{t1.1}$, if $(\rho, \mathbf{u}, \mathbf{B})$ is a strong solution to the Cauchy problem \eqref{a1}--\eqref{a3} satisfying
\begin{align}\label{3.1}
\sup_{\mathbb{R}^2\times[0,T]}\rho\le2\hat{\rho},\ \ \frac{1}{(2\mu+\lambda)^2}\int_{0}^{T}
\|P-P(\tilde{\rho})\|_{L^4}^4\mathrm{d}t\le2,
\end{align}
then one has that
\begin{align}\label{3.2}
\sup_{\mathbb{R}^2\times[0,T]}\rho\le\frac{7}{4}\hat{\rho},\ \ \frac{1}{(2\mu+\lambda)^2}\int_{0}^{T}
\|P-P(\tilde{\rho})\|_{L^4}^4\mathrm{d}t\le1.
\end{align}
\end{proposition}

Before proving Proposition \ref{p3.1}, we establish some necessary \textit{a priori} estimates, see Lemmas \ref{l3.1}--\ref{l3.4} below. We begin with the elementary energy estimate of $(\rho, \mathbf{u}, \mathbf{B})$.
\begin{lemma}\label{l3.1}
It holds that
\begin{align}\label{3.3}
\sup_{0\le t\le T}\int\bigg(\frac{1}{2}\rho |\mathbf{u}|^2+\frac{1}{2} |\mathbf{B}|^2+G(\rho)\bigg)\mathrm{d}\mathbf{x}+\int_0^T\left[\mu\|\nabla \mathbf{u}\|_{L^2}^2+(\mu+\lambda)\|\divv \mathbf{u}\|_{L^2}^2+\nu\|\nabla \mathbf{B}\|_{L^2}^2\right]\mathrm{d}t
\leq C_0.
\end{align}
\end{lemma}
\begin{proof}
By $\eqref{a1}_1$ and the definition of $G(\rho)$ in \eqref{1.4}, we get that
\begin{equation}\label{3.4}
(G(\rho))_t+\divv (G(\rho)\mathbf{u})+(P-P(\tilde{\rho}))\divv \mathbf{u}=0.
\end{equation}
Multiplying $\eqref{a1}_2$ by $\mathbf{u}$ and $\eqref{a1}_3$ by $\mathbf{B}$, respectively, adding the summation to \eqref{3.4}, and integrating (by parts) the resultant over $\mathbb{R}^2$, we arrive at
\begin{align}\label{3.5}
\frac{\mathrm{d}}{\mathrm{d}t}\int\bigg(\frac{1}{2}\rho|\mathbf{u}|^2+\frac{1}{2} |\mathbf{B}|^2+G(\rho)\bigg)\mathrm{d}\mathbf{x}
+\int\left[\mu|\nabla\mathbf{u}|^2+(\mu+\lambda)(\divv\mathbf{u})^2
+\nu|\nabla\mathbf{B}|^2\right]\mathrm{d}\mathbf{x}
=0,
\end{align}
which yields \eqref{3.3} after integrating \eqref{3.5} with respect to $t$ over $(0,T)$.
\end{proof}

The next lemma provides a time-independent estimate for $\nabla \mathbf{u}$ and $\nabla \mathbf{B}$ in $L^\infty(0, T;L^2)$.
\begin{lemma}\label{l3.2}
Let \eqref{3.1} be satisfied, then there exists a positive number $D_2$ depending only on $\tilde{\rho}$, $\hat{\rho}$, $a$, $\gamma$, $\nu$, and $\mu$ such that
\begin{align}\label{3.6}
 &\sup_{0\leq t\leq T}\int\left[\mu|\nabla \mathbf{u}|^2+(\mu+\lambda)(\divv \mathbf{u})^2+\nu|\nabla \mathbf{B}|^2\right]\mathrm{d}\mathbf{x}\notag\\
 &\quad +\int_0^T\big(\|\sqrt{\rho}{\dot{\mathbf{u}}}\|_{L^2}^2+\|\nabla^2 \mathbf{B}\|_{L^2}^2+\|\mathbf{B}_t\|_{L^2}^2\big)
\mathrm{d}t\leq(2+M)^{\exp\big\{2D_2(1+C_0)^6\big\}}
\end{align}
provided that $\lambda$ satisfies \eqref{lam} with $D\geq D_2$.
\end{lemma}
\begin{proof}
Owing to $\eqref{a1}_1$, we rewrite $\eqref{a1}_2$ as
\begin{equation}\label{3.7}
\rho\mathbf{u}_t+\rho\mathbf{u}\cdot\nabla\mathbf{u}-\mu\Delta\mathbf{u}-(\mu+\lambda)\nabla\divv \mathbf{u}+\nabla P=\mathbf{B}\cdot\nabla\mathbf{B}-\frac12\nabla|\mathbf{B}|^2.
\end{equation}
Multiplying \eqref{3.7} by $\mathbf{u}_t$ and integration by parts, we obtain that
\begin{align}\label{3.8}
&\frac{1}{2}\frac{\mathrm{d}}{\mathrm{d}t}\int\left[\mu|\nabla \mathbf{u}|^2+(\mu+\lambda)(\divv \mathbf{u})^2
-|\mathbf{B}|^2\divv
\mathbf{u}+2\mathbf{B}\cdot\nabla\mathbf{u}\cdot\mathbf{B}\right]\mathrm{d}\mathbf{x}+\int\rho |\dot{\mathbf{u}}|^2\mathrm{d}\mathbf{x}\notag\\
&=-\int\mathbf{u}_t\cdot\nabla P\mathrm{d}\mathbf{x}+\int\rho \dot{\mathbf{u}}\cdot(\mathbf{u}\cdot\nabla)\mathbf{u}\mathrm{d}\mathbf{x}
+\int\big(\mathbf{B}_t\cdot\nabla\mathbf{u}\cdot\mathbf{B}
+\mathbf{B}\cdot\nabla\mathbf{u}\cdot\mathbf{B}_t
-\mathbf{B}\cdot\mathbf{B}_t\divv\mathbf{u}\big)\mathrm{d}\mathbf{x}.
\end{align}
Moreover, it follows from $\eqref{a1}_3$ that
\begin{align*}
\int|\mathbf{B}\cdot\nabla\mathbf{u}-\mathbf{u}\cdot\nabla\mathbf{B}
-\mathbf{B}\divv\mathbf{u}|^2\mathrm{d}\mathbf{x}
&=\int|\mathbf{B}_t-\nu\Delta\mathbf{B}|^2\mathrm{d}\mathbf{x}\notag\\
&=\frac{\mathrm{d}}{\mathrm{d}t}\int\nu|\nabla \mathbf{B}|^2\mathrm{d}\mathbf{x}+
\int\big(\nu^2|\nabla^2 \mathbf{B}|^2+|\mathbf{B}_t|^2\big)\mathrm{d}\mathbf{x},
\end{align*}
which combined with \eqref{3.8} yields that
\begin{align}\label{3.10}
&\frac{1}{2}\frac{\mathrm{d}}{\mathrm{d}t}\int\left[\mu|\nabla \mathbf{u}|^2+(\mu+\lambda)(\divv \mathbf{u})^2+2\nu|\nabla \mathbf{B}|^2
-|\mathbf{B}|^2\divv\mathbf{u}+2\mathbf{B}\cdot\nabla\mathbf{u}\cdot\mathbf{B}\right]
\mathrm{d}\mathbf{x}
\notag\\&\quad+\int\big(\rho |\dot{\mathbf{u}}|^2+\nu^2|\nabla^2 \mathbf{B}|^2+|\mathbf{B}_t|^2\big)\mathrm{d}\mathbf{x}\notag\\
 &=-\int\mathbf{u}_t\cdot\nabla P\mathrm{d}\mathbf{x}+\int\rho \dot{\mathbf{u}}\cdot(\mathbf{u}\cdot\nabla)\mathbf{u}\mathrm{d}\mathbf{x}
 +\int(\mathbf{B}_t\cdot\nabla\mathbf{u}\cdot\mathbf{B}+\mathbf{B}
 \cdot\nabla\mathbf{u}\cdot\mathbf{B}_t
-\mathbf{B}\cdot\mathbf{B}_t\divv\mathbf{u})\mathrm{d}\mathbf{x}
 \notag\\
&\quad+\int|\mathbf{B}\cdot\nabla\mathbf{u}-\mathbf{u}\cdot\nabla\mathbf{B}
-\mathbf{B}\divv\mathbf{u}|^2\mathrm{d}\mathbf{x}
\triangleq\sum_{i=1}^{4}\mathcal{I}_i.
\end{align}

Next, one needs to bound each $\mathcal{I}_i$. Note that
\begin{align}\label{3.11}
\tilde{\rho}\int|\mathbf{u}|^2\mathrm{d}\mathbf{x}
&=\int\rho|\mathbf{u}|^2\mathrm{d}\mathbf{x}
+\int(\tilde{\rho}-\rho)|\mathbf{u}|^2\mathrm{d}\mathbf{x}\notag\\
&\leq\|\sqrt{\rho}\mathbf{u}\|_{L^2}^2
+\|\rho-\tilde{\rho}\|_{L^2}\|\mathbf{u}\|_{L^4}^2\notag\\
&\leq\|\sqrt{\rho}\mathbf{u}\|_{L^2}^2
+C\|\rho-\tilde{\rho}\|_{L^2}\|\mathbf{u}\|_{L^2}\|\nabla\mathbf{u}\|_{L^2}\notag
\\&\leq2C_0
+CC_0^\frac{1}{2}\|\mathbf{u}\|_{L^2}\|\nabla\mathbf{u}\|_{L^2},
\end{align}
due to Gagliardo--Nirenberg inequality, Lemma $\ref{l3.1}$, and \eqref{1.10}.
This together with the fact $\tilde{\rho}>0$ and Cauchy--Schwarz inequality shows that
\begin{equation}\label{3.12}
\|\mathbf{u}\|_{L^2}^2 \leq C(\tilde{\rho},\hat{\rho})
C_0\big(1+\|\nabla\mathbf{u}\|_{L^2}^2\big),
\end{equation}
and furthermore,
\begin{equation}\label{3.13}
\|\mathbf{u}\|_{H^1}\leq C(\tilde{\rho},\hat{\rho})(1+C_0)^\frac12(1
+\|\nabla\mathbf{u}\|_{L^2}).
\end{equation}

According to $\eqref{a1}_1$ and the definition of effective viscous flux in \eqref{1.6}, one gets that
\begin{align}\label{3.14}
&\mathcal{I}_1=\int (P-P(\tilde{\rho}))\divv\mathbf{u}_t\mathrm{d}\mathbf{x}\notag\\
&=\frac{\mathrm{d}}{\mathrm{d}t}\int(P-P(\tilde{\rho}))\divv\mathbf{u} \mathrm{d}\mathbf{x}+\int\divv\mathbf{u}P'(\rho)\rho\divv\mathbf{u}
\mathrm{d}\mathbf{x}+\int\divv\mathbf{u}P'(\rho)\mathbf{u}\cdot\nabla\rho\mathrm{d}\mathbf{x}\notag\\
&=\frac{\mathrm{d}}{\mathrm{d}t}\int(P-P(\tilde{\rho}))\divv\mathbf{u}\mathrm{d}\mathbf{x}+
\int(\divv\mathbf{u})^2P'(\rho)\rho\mathrm{d}\mathbf{x}+
\int\mathbf{u}\cdot\nabla(P-P(\tilde{\rho}))\divv\mathbf{u}\mathrm{d}\mathbf{x}\notag\\
&=\frac{\mathrm{d}}{\mathrm{d}t}\int(P-P(\tilde{\rho}))\divv\mathbf{u}\mathrm{d}\mathbf{x}+
\int(\divv\mathbf{u})^2(P'(\rho)\rho-P+P(\tilde{\rho}))\mathrm{d}\mathbf{x}-
\int(P-P(\tilde{\rho}))\mathbf{u}\cdot\nabla\divv\mathbf{u}\mathrm{d}\mathbf{x}\notag\\
&=\frac{\mathrm{d}}{\mathrm{d}t}\int(P-P(\tilde{\rho}))\divv\mathbf{u}\mathrm{d}\mathbf{x}+
\int(\divv\mathbf{u})^2(P'(\rho)\rho-P+P(\tilde{\rho}))\mathrm{d}\mathbf{x}\notag\\
&\quad-\frac{1}{2\mu+\lambda}
\int(P-P(\tilde{\rho}))\mathbf{u}\cdot\nabla\Big(F+P-P(\tilde{\rho})+\frac12|\mathbf{B}|^2\Big)\mathrm{d}\mathbf{x}\notag\\
&=\frac{\mathrm{d}}{\mathrm{d}t}\int(P-P(\tilde{\rho}))\divv\mathbf{u}\mathrm{d}\mathbf{x}
+\int(\divv\mathbf{u})^2(P'(\rho)\rho-P+P(\tilde{\rho}))\mathrm{d}\mathbf{x}
+\frac{1}{4\mu+2\lambda}
\int(P-P(\tilde{\rho}))^2\divv\mathbf{u}\mathrm{d}\mathbf{x}\notag\\
&\quad
-\frac{1}{2\mu+\lambda}
\int (P-P(\tilde{\rho}))\mathbf{u}\cdot\nabla F\mathrm{d}\mathbf{x}
-\frac{1}{2\mu+\lambda}
\int (P-P(\tilde{\rho}))\mathbf{u}\cdot\nabla \mathbf{B}\cdot\mathbf{B}\mathrm{d}\mathbf{x}
\notag\\
&\leq\frac{\mathrm{d}}{\mathrm{d}t}\int(P-P(\tilde{\rho}))
\divv\mathbf{u}\mathrm{d}\mathbf{x}
+\frac{1}{8}\|\sqrt{\rho}\dot{\mathbf{u}}\|_{L^2}^2
+\frac{\nu^2}{8}\|\nabla^2\mathbf{B}\|_{L^2}^2
+C\|\nabla\mathbf{B}\|_{L^2}^2\big(1+\|\nabla\mathbf{B}\|_{L^2}^2\big)
\notag\\
&\quad
+C(1+C_0)^2\|\nabla\mathbf{u}\|_{L^{2}}^{2}
\big(1+\|\nabla\mathbf{u}\|_{L^{2}}^{2}\big)
+\frac{C}{(2\mu+\lambda)^2}\|P-P(\tilde{\rho})\|_{L^4}^4,
\end{align}
where we have used \eqref{3.12} and the following estimate
\begin{align*}
&\frac{1}{2\mu+\lambda}\left|
\int(P-P(\tilde{\rho}))\mathbf{u}\cdot\nabla F\mathrm{d}\mathbf{x}\right|
+\frac{1}{2\mu+\lambda}\left|
\int(P-P(\tilde{\rho}))\mathbf{u}\cdot\nabla \mathbf{B}\cdot\mathbf{B}\mathrm{d}\mathbf{x}\right|\notag\\
&\leq\frac{C}{2\mu+\lambda}\|P-P(\tilde{\rho})\|_{L^\infty}
(\|\mathbf{u}\|_{L^2}
\|\nabla F\|_{L^2}+\|\mathbf{u}\|_{L^4}\|\nabla \mathbf{B}\|_{L^4}\|\mathbf{B}\|_{L^2})\notag\\
&\leq\frac{C(\tilde{\rho},\hat{\rho})C_0^\frac12}{2\mu+\lambda}
\Big[(1+\|\nabla\mathbf{u}\|_{L^2}^2)^\frac12
\big(\|\sqrt{\rho}\dot{\mathbf{u}}\|_{L^2}+\|\mathbf{B}\cdot\nabla\mathbf{B}
\|_{L^2}\big)
+\|\mathbf{u}\|_{L^2}^\frac12
\|\nabla\mathbf{u}\|_{L^2}^\frac12\|\nabla\mathbf{B}\|_{L^2}^\frac12
\|\nabla^2\mathbf{B}\|_{L^2}^\frac12\Big]
\notag\\
&\leq\frac{CC_0^\frac12}{2\mu+\lambda}
\Big[\big(1+\|\nabla\mathbf{u}\|_{L^2}^2\big)^\frac12
\Big(\|\sqrt{\rho}\dot{\mathbf{u}}\|_{L^2}+C_0^\frac14\|\nabla\mathbf{B}\|_{L^2}
\|\nabla^2\mathbf{B}\|_{L^2}^\frac12\Big)
+C_0^\frac14\big(1+\|\nabla\mathbf{u}\|_{L^2}^2\big)^\frac12
\|\nabla\mathbf{B}\|_{L^2}^\frac12\|\nabla^2\mathbf{B}\|_{L^2}^\frac12\Big]
\notag\\
&\leq\frac{1}{8}\|\sqrt{\rho}\dot{\mathbf{u}}\|_{L^2}^2
+\frac{\nu^2}{8}\|\nabla^2\mathbf{B}\|_{L^2}^2
+C\|\nabla\mathbf{B}\|_{L^2}^2\big(1+\|\nabla\mathbf{B}\|_{L^2}^2\big)
+CC_0^\frac32\big(1+\|\nabla\mathbf{u}\|_{L^{2}}^{2}\big).
\end{align*}

Next, we consider $\mathcal{I}_2$ in two cases.

\textbf{Case 1:} If $\|\nabla\mathbf{u}\|_{L^{2}}\leq1$, \eqref{3.13} simplifies to
\begin{equation*}
\|\mathbf{u}\|_{H^1}\leq C(\tilde{\rho},\hat{\rho})(1+C_0)^\frac12.
\end{equation*}
Then it follows from \eqref{3.1}, Lemma $\ref{E0}$, Cauchy--Schwarz inequality, Gagliardo--Nirenberg inequality, and H\"older's inequality that
\begin{align}\label{3.15}
&\mathcal{I}_2=\int\rho \dot{\mathbf{u}}\cdot(\mathbf{u}\cdot\nabla)\mathbf{u}\mathrm{d}\mathbf{x}\leq C\|\sqrt{\rho}\dot{\mathbf{u}}\|_{L^2}\|\mathbf{u}\|_{L^{4}}\|\nabla\mathbf{u}\|_{L^{4}}\notag\\
&\leq C\|\sqrt{\rho}\dot{\mathbf{u}}\|_{L^2}\|\mathbf{u}\|_{L^{2}}^{\frac{1}{2}}
\|\nabla\mathbf{u}\|_{L^{2}}^{\frac{1}{2}}\Big[\Big(\|\sqrt{\rho}\dot{\mathbf{u}}\|_{L^2}
^{\frac{1}{2}}+\|\mathbf{B}\cdot\nabla\mathbf{B}\|_{L^2}^{\frac{1}{2}}\Big)
\|\nabla\mathbf{u}\|_{L^2}^\frac{1}{2}
+\frac{1}{2\mu+\lambda}\|\mathbf{B}\|_{L^{8}}^2+\frac{1}{2\mu+\lambda}\|P-P(\tilde{\rho})\|_{L^4}\notag\\
&\quad+
\frac{1}{2\mu+\lambda}\Big(\|\sqrt{\rho}\dot{\mathbf{u}}\|_{L^2}^{\frac{1}{2}}
+\|\mathbf{B}\cdot\nabla\mathbf{B}\|_{L^2}^{\frac{1}{2}}\Big)
\Big(\|P-P(\tilde{\rho})\|_{L^2}^{\frac{1}{2}}+\|\mathbf{B}\|_{L^4}\Big)\Big]\notag\\
&\leq \frac{1}{8}\|\sqrt{\rho}\dot{\mathbf{u}}\|_{L^2}^2
+\frac{\nu^2}{8}\|\nabla^2\mathbf{B}\|_{L^2}^2
+C\|\nabla\mathbf{B}\|_{L^2}^4+C(1+C_0)^5\big(1+\|\nabla \mathbf{u}\|_{L^2}^2\big)
+\frac{C}{(2\mu+\lambda)^4}\|P-P(\tilde{\rho})\|_{L^4}^4,
\end{align}
where in the last inequality we have used
\begin{align*}
\|\mathbf{B}\|_{L^{8}}^2\leq\|\mathbf{B}\|_{L^{\infty}}\|\mathbf{B}\|_{L^{4}}
\leq C\|\mathbf{B}\|_{L^{4}}^\frac32\|\nabla\mathbf{B}\|_{L^{4}}^\frac12
\leq C\|\mathbf{B}\|_{L^{2}}^\frac34\|\nabla\mathbf{B}\|_{L^{2}}\|\nabla^2\mathbf{B}\|_{L^{2}}^\frac14.
\end{align*}

\textbf{Case 2:} If $\|\nabla\mathbf{u}\|_{L^{2}}\ge1$, noting that
\begin{equation*}
\ln(2+bc)\le b\ln(2+c),~\mathrm{for}~b,c\ge1,
\end{equation*}
then one deduces from Lemmas $\ref{log}$, $\ref{E0}$, $\ref{l3.1}$, and \eqref{3.13} that
\begin{align}\label{3.16}
&\mathcal{I}_2=\int\rho \dot{\mathbf{u}}\cdot(\mathbf{u}\cdot\nabla)\mathbf{u}\mathrm{d}\mathbf{x}\leq C\|\sqrt{\rho}\dot{\mathbf{u}}\|_{L^2}\|\sqrt{\rho}\mathbf{u}\|_{L^{4}}
\|\nabla\mathbf{u}\|_{L^{4}}\notag\\
&\leq C\|\sqrt{\rho}\dot{\mathbf{u}}\|_{L^2}
(1+\|\sqrt\rho\mathbf{u}\|_{L^2})^\frac12\|\mathbf{u}\|_{H^1}^\frac12
\ln^\frac14\left(2+\|\mathbf{u}\|_{H^1}^2\right)
\Big[\Big(\|\sqrt{\rho}\dot{\mathbf{u}}\|_{L^2}^{\frac{1}{2}}
+\|\mathbf{B}\cdot\nabla\mathbf{B}\|_{L^2}^{\frac{1}{2}}\Big)
\|\nabla\mathbf{u}\|_{L^2}^\frac{1}{2}\notag\\
&\quad
+\frac{1}{2\mu+\lambda}\|\mathbf{B}\|_{L^{8}}^2
+\frac{1}{2\mu+\lambda}\|P-P(\tilde{\rho})\|_{L^4}+
\frac{1}{2\mu+\lambda}\Big(\|\sqrt{\rho}\dot{\mathbf{u}}\|_{L^2}^{\frac{1}{2}}
+\|\mathbf{B}\cdot\nabla\mathbf{B}\|_{L^2}^{\frac{1}{2}}\Big)
\Big(\|P-P(\tilde{\rho})\|_{L^2}^{\frac{1}{2}}+\|\mathbf{B}\|_{L^4}\Big)\Big]\notag\\
&\leq C\|\sqrt{\rho}\dot{\mathbf{u}}\|_{L^2}\Big(1+C_0^{\frac{1}{2}}\Big)^{\frac{1}{2}}(1+C_0)^{\frac14}
(1+\|\nabla\mathbf{u}\|_{L^{2}})^{\frac{1}{2}}(1+C_0)^{\frac14}\ln^{\frac{1}{4}}\left(2+\|\nabla\mathbf{u}\|_{L^2}^2\right)
\Big[\|\sqrt{\rho}\dot{\mathbf{u}}\|_{L^2}^{\frac{1}{2}}
\|\nabla\mathbf{u}\|_{L^2}^\frac{1}{2}\notag\\
&\quad+C_0^\frac18\|\nabla\mathbf{B}\|_{L^2}^\frac12\|\nabla^2\mathbf{B}\|_{L^2}^\frac14\|\nabla\mathbf{u}\|_{L^2}^\frac12
+\frac{1}{2\mu+\lambda}C_0^\frac18\|\mathbf{B}\|_{L^{4}}\|\nabla\mathbf{B}\|_{L^{2}}^\frac12\|\nabla^2\mathbf{B}\|_{L^{2}}^\frac14
+\frac{1}{2\mu+\lambda}\|P-P(\tilde{\rho})\|_{L^4}\notag\\
&\quad
+\frac{1}{2\mu+\lambda}\Big(\|\sqrt{\rho}\dot{\mathbf{u}}\|_{L^2}^{\frac{1}{2}}
+C_0^\frac18\|\nabla\mathbf{B}\|_{L^2}^\frac12
\|\nabla^2\mathbf{B}\|_{L^2}^\frac14\Big)
\Big(C_0^{\frac14}+\|\mathbf{B}\|_{L^4}\Big)
\Big]\notag\\
&\leq \frac{1}{8}\|\sqrt{\rho}\dot{\mathbf{u}}\|_{L^2}^2
+\frac{\nu^2}{8}\|\nabla^2\mathbf{B}\|_{L^{2}}^2
+C\|\nabla\mathbf{B}\|_{L^2}^2
\big(1+\|\nabla\mathbf{B}\|_{L^2}^2+\|\mathbf{B}\|_{L^{4}}^4\big)
+\frac{C}{(2\mu+\lambda)^4}\|P-P(\tilde{\rho})\|_{L^4}^4\notag\\
&\quad+
C(1+C_0)^5
\big(\|\nabla\mathbf{u}\|_{L^{2}}^2+\|\nabla\mathbf{B}\|_{L^2}^2\big)
\big(1+\|\nabla\mathbf{u}\|_{L^{2}}^2+\|\mathbf{B}\|_{L^{4}}^4\big)
\ln\left(2+\|\nabla\mathbf{u}\|_{L^2}^2\right),
\end{align}
where in the third inequality one has used
\begin{equation*}
\|\mathbf{B}\cdot\nabla\mathbf{B}\|_{L^2}\leq\|\mathbf{B}\|_{L^4}\|\nabla\mathbf{B}\|_{L^4}
\leq C\|\mathbf{B}\|_{L^2}^\frac12\|\nabla\mathbf{B}\|_{L^2}\|\nabla^2\mathbf{B}\|_{L^2}^\frac12
\leq CC_0^\frac14\|\nabla\mathbf{B}\|_{L^2}\|\nabla^2\mathbf{B}\|_{L^2}^\frac12.
\end{equation*}

Using Gagliardo--Nirenberg inequality and \eqref{3.12}, one finds that
\begin{align}\label{3.17}
 \mathcal{I}_3&\leq C\|\mathbf{B}\|_{L^\infty}\|\mathbf{B}_t\|_{L^2}\|\nabla\mathbf{u}\|_{L^2}\leq CC_0^\frac18\|\nabla\mathbf{B}\|_{L^2}^\frac12\|\nabla^2\mathbf{B}\|_{L^2}^\frac14\|\mathbf{B}_t\|_{L^2}\|\nabla\mathbf{u}\|_{L^2}\notag\\
 &\leq \frac14\|\mathbf{B}_t\|_{L^2}^2+\frac{\nu^2}{8}\|\nabla^2\mathbf{B}\|_{L^{2}}^2
 +C\|\nabla\mathbf{B}\|_{L^2}^4+CC_0^\frac12\|\nabla\mathbf{u}\|_{L^2}^4.
\end{align}
For $\mathcal{I}_4$, one derives from integration by parts that
\begin{align}\label{3.18}
 \mathcal{I}_4&\leq C\|\mathbf{B}\|_{L^\infty}^2\|\nabla\mathbf{u}\|_{L^2}^2
 +C\int(\mathbf{u}\cdot\nabla)B^j(\mathbf{u}\cdot\nabla)B^j\mathrm{d}\mathbf{x}\notag\\
 &\leq C\big[\|\mathbf{B}\|_{L^\infty}^2\|\nabla\mathbf{u}\|_{L^2}^2
 +\|\mathbf{B}\|_{L^4}\|\nabla\mathbf{B}\|_{L^8}\|\mathbf{u}\|_{L^8}\|\nabla\mathbf{u}\|_{L^2}
 +\|\mathbf{B}\|_{L^4}\|\mathbf{u}\|_{L^8}^2\|\nabla^2\mathbf{B}\|_{L^2}\big]
 \notag\\
 &\leq C C_0^\frac14\Big[\|\nabla\mathbf{B}\|_{L^2}\|\nabla^2\mathbf{B}\|_{L^2}^\frac12\|\nabla\mathbf{u}\|_{L^2}^2
 +\|\mathbf{B}\|_{L^4}\|\nabla\mathbf{u}\|_{L^2}
 \big(1+\|\nabla\mathbf{u}\|_{L^2}^2\big)^\frac12\Big(
 \|\nabla\mathbf{B}\|_{L^2}^\frac14\|\nabla^2\mathbf{B}\|_{L^2}^\frac34
+\|\nabla^2\mathbf{B}\|_{L^2}\Big)\Big]\notag\\
 &\leq C(1+C_0)^2\|\nabla\mathbf{u}\|_{L^2}^2\big(1+\|\nabla\mathbf{u}\|_{L^{2}}^2\big)
 +\frac{\nu^2}{8}\|\nabla^2\mathbf{B}\|_{L^{2}}^2+C\|\nabla\mathbf{B}\|_{L^2}^2
\big(1+\|\nabla\mathbf{u}\|_{L^2}^2+\|\nabla\mathbf{B}\|_{L^2}^2\big).
\end{align}

Thus, substituting \eqref{3.14}--\eqref{3.18} into \eqref{3.10}, one concludes that
\begin{align}\label{3.22}
&\frac{1}{2}\frac{\mathrm{d}}{\mathrm{d}t}\int\left[\mu|\nabla \mathbf{u}|^2+(\mu+\lambda)(\divv \mathbf{u})^2+2\nu|\nabla \mathbf{B}|^2-2(P-P(\tilde{\rho}))\divv\mathbf{u}
-|\mathbf{B}|^2\divv
\mathbf{u}+2\mathbf{B}\cdot\nabla\mathbf{u}\cdot\mathbf{B}\right]
\mathrm{d}\mathbf{x}
\notag\\&\quad+\int\big(\rho |\dot{\mathbf{u}}|^2+\nu^2|\nabla^2 \mathbf{B}|^2+|\mathbf{B}_t|^2\big)\mathrm{d}\mathbf{x}\notag\\
&\leq
C(1+C_0)^5
\big(\|\nabla\mathbf{u}\|_{L^{2}}^2+\|\nabla\mathbf{B}\|_{L^2}^2\big)
\big(1+\|\nabla\mathbf{u}\|_{L^{2}}^2+\|\mathbf{B}\|_{L^{4}}^4\big)
\ln\left(2+\|\nabla\mathbf{u}\|_{L^2}^2\right)
\notag\\
&\quad
+C\|\nabla\mathbf{B}\|_{L^2}^2\big(1+\|\nabla\mathbf{B}\|_{L^2}^2\big)
+\frac{C}{(2\mu+\lambda)^2}\|P-P(\tilde{\rho})\|_{L^4}^4,
\end{align}
where we observe that
\begin{equation}\label{3.20}
 \bigg|\int(|\mathbf{B}|^2\divv
\mathbf{u}-2\mathbf{B}\cdot\nabla\mathbf{u}\cdot\mathbf{B})\mathrm{d}\mathbf{x}\bigg|\leq C\|\mathbf{B}\|_{L^4}^2\|\nabla\mathbf{u}\|_{L^2}
  \leq C\|\mathbf{B}\|_{L^4}^4+ \frac{\mu}{8}\|\nabla\mathbf{u}\|_{L^2}^2.
\end{equation}
Moreover, multiplying $\eqref{a1}_3$ by $4|\mathbf{B}|^2\mathbf{B}$ and integrating the resultant over $\mathbb{R}^2$ yields that
\begin{align}\label{3.19}
\frac{\mathrm{d}}{\mathrm{d}t}\|\mathbf{B}\|_{L^4}^4
+4\nu\||\mathbf{B}||\nabla\mathbf{B}|\|_{L^2}^2
&\leq -2\nu\|\nabla|\mathbf{B}|^2\|_{L^2}^2+C\|\nabla\mathbf{u}\|_{L^2}\||\mathbf{B}|^2\|_{L^4}^2 \notag \\
 & \leq -2\nu\|\nabla|\mathbf{B}|^2\|_{L^2}^2 +C\|\nabla\mathbf{u}\|_{L^2}\||\mathbf{B}|^2\|_{L^2}\|\nabla|\mathbf{B}|^2\|_{L^2}\notag\\
  &\leq
  -\nu\|\nabla|\mathbf{B}|^2\|_{L^2}^2
  +C\|\mathbf{B}\|_{L^4}^4\|\nabla\mathbf{u}\|_{L^2}^2.
\end{align}
Now we define an auxiliary functional as
\begin{align}\label{3.23}
\mathcal{E}_1(t)\triangleq
&\int\left[\mu|\nabla \mathbf{u}|^2+(\mu+\lambda)(\divv \mathbf{u})^2+2\nu|\nabla \mathbf{B}|^2-2(P-P(\tilde{\rho}))\divv\mathbf{u}
-|\mathbf{B}|^2\divv
\mathbf{u}+2\mathbf{B}\cdot\nabla\mathbf{u}\cdot\mathbf{B}\right]\mathrm{d}\mathbf{x}
\notag\\
&+
\int\bigg(\frac{1}{2}\rho |\mathbf{u}|^2+\frac{1}{2} |\mathbf{B}|^2+G(\rho)\bigg)\mathrm{d}\mathbf{x},
\end{align}
which along with \eqref{3.1} and \eqref{3.20} indicates that there exists $D_1=D_1(\tilde{\rho},\hat{\rho})>0$ such that
\begin{align}\label{3.24}
\mathcal{E}_1(t)\thicksim\int\bigg(\frac{1}{2}\rho |\mathbf{u}|^2+\frac{1}{2} |\mathbf{B}|^2+G(\rho)
+\mu|\nabla \mathbf{u}|^2+(\mu+\lambda)(\divv \mathbf{u})^2+2\nu|\nabla \mathbf{B}|^2+|\mathbf{B}|^4\bigg)\mathrm{d}\mathbf{x}
\end{align}
provided $\lambda\ge D_1$.

Setting
\begin{align}\label{3.25}
f_1(t)\triangleq2+\mathcal{E}_1(t),~~g_1(t)\triangleq (1+C_0)^5\big(\|\nabla\mathbf{u}\|_{L^2}^2+\|\nabla \mathbf{B}\|_{L^2}^2\big)
+\frac{1}{(2\mu+\lambda)^2}\|P-P(\tilde{\rho})\|_{L^4}^4
\end{align}
and adding \eqref{3.5} into \eqref{3.22}, we then deduce from \eqref{3.19} and \eqref{3.24} that
\begin{equation*}
    f'_{1}(t)\leq Cg_{1}(t)f_{1}(t)\ln f_{1}(t),
\end{equation*}
and hence,
\begin{equation}\label{3.26}
    \big(\ln f_1(t)\big)'\leq Cg_1(t)\ln f_1(t).
\end{equation}
From \eqref{3.3} and Gronwall's inequality, there is a positive constant $D_2=D_2(\tilde{\rho},\hat{\rho},a,\gamma,\nu,\mu)\geq D_1$ such that
\begin{equation}\label{3.27}
    \sup_{0\leq t\leq T}\int\left[\mu|\nabla \mathbf{u}|^2+(\mu+\lambda)(\divv \mathbf{u})^2+\nu|\nabla \mathbf{B}|^2+|\mathbf{B}|^4\right]\mathrm{d}\mathbf{x}\leq
(2+M)^{\exp\big\{D_2(1+C_0)^6\big\}}
\end{equation}
provided $\lambda\ge D_2$.
Integrating \eqref{3.22} with respect to $t$ over $(0,T)$ together with \eqref{3.3} and \eqref{3.27} yields that
\begin{align*}
&\int_0^T\big(\|\sqrt{\rho}{\dot{\mathbf{u}}}\|_{L^2}^2+\|\nabla^2 \mathbf{B}\|_{L^2}^2+
\|\mathbf{B}_t\|_{L^2}^2\big)
\mathrm{d}t\notag\\
&\leq C(1+M)+C(1+C_0)^5C_0(2+M)^{\exp\big\{D_2(1+C_0)^6\big\}}
\ln\bigg\{(2+M)^{\exp\big\{\frac{3}{2}D_2(1+C_0)^6\big\}}\bigg\}\notag\\
&\leq(2+M)^{\exp\big\{\frac{7}{4}D_2(1+C_0)^6\big\}}
\end{align*}
provided that $\lambda$ satisfies \eqref{lam} with $D\geq D_2$, which along with \eqref{3.27} implies the desired \eqref{3.6}.
\end{proof}

Next, we give the bound of $\frac{1}{2\mu+\lambda}\int_{0}^{T}\|P-P(\tilde{\rho})\|_{L^4}^4\mathrm{d}t$.
\begin{lemma}\label{l3.3}
Let \eqref{3.1} be satisfied, then it holds that
\begin{align}\label{3.29}
\frac{1}{2\mu+\lambda}\int_{0}^{T}\|P-P(\tilde{\rho})\|_{L^4}^4\mathrm{d}t\le
(2+M)^{\exp\big\{3D_2(1+C_0)^6\big\}}
\end{align}
provided that $\lambda$ satisfies \eqref{lam} with $D\geq 2D_2$.
\end{lemma}
\begin{proof}
It follows from $\eqref{a1}_1$ and $P(\rho)=a\rho^\gamma$ that
\begin{equation}\label{3.30}
  (P-P(\tilde{\rho}))_t+\mathbf{u}\cdot\nabla(P-P(\tilde{\rho}))+\gamma (P-P(\tilde{\rho}))\divv\mathbf{u}+\gamma P(\tilde{\rho})\divv\mathbf{u}=0.
\end{equation}
Multiplying $\eqref{3.30}$ by $3(P-P(\tilde{\rho}))^2$ and integrating the resulting equality over $\mathbb{R}^2$, we obtain that
\begin{align*}
&\frac{3\gamma-1}{2\mu+\lambda}\left\|P-P(\tilde{\rho})\right\|_{L^4}^4\notag\\
&=-\frac{\mathrm{d}}{\mathrm{d}t}\int(P-P(\tilde{\rho}))^3\mathrm{d}\mathbf{x}
-\frac{3\gamma-1}{2(2\mu+\lambda)}\int(P-P(\tilde{\rho}))^3\big(2F+|\mathbf{B}|^2\big)
\mathrm{d}\mathbf{x}
-3\gamma P(\tilde{\rho})\int(P-P(\tilde{\rho}))^2\divv\mathbf{u}\mathrm{d}\mathbf{x}\notag\\
&\leq-\frac{\mathrm{d}}{\mathrm{d}t}\int(P-P(\tilde{\rho}))^3\mathrm{d}\mathbf{x}+
\frac{3\gamma-1}{4(2\mu+\lambda)}\left\|P-P(\tilde{\rho})\right\|_{L^4}^4
+\frac{C}{2\mu+\lambda}\big(\|F\|_{L^4}^4+\|\mathbf{B}\|_{L^8}^8\big) +C(2\mu+\lambda)\left\|\divv\mathbf{u}\right\|_{L^2}^2\notag.
\end{align*}
Integrating the above inequality over $(0,T)$, we infer from Lemmas $\ref{E0}$, $\ref{l3.1}$, $\ref{l3.2}$, and \eqref{3.27} that
\begin{align*}
&\frac{1}{2\mu+\lambda}\int_{0}^{T}\|P-P(\tilde{\rho})\|_{L^{4}}^{4}\mathrm{d}t\notag\\
&\leq \frac{C}{2\mu+\lambda}\int_0^T\big(\|F\|_{L^2}^2\|\nabla F\|_{L^2}^2+\|\mathbf{B}\|_{L^\infty}^4\|\mathbf{B}\|_{L^4}^4\big)\mathrm{d}t+
C\sup_{0\leq t\leq T}\left\|P-P(\tilde{\rho})\right\|_{L^3}^3+CC_0\notag\\
&\leq \frac{C}{2\mu+\lambda}\int_0^T\left[(2\mu+\lambda)^2
\|\divv\mathbf{u}\|_{L^2}^2+\|P-P(\tilde{\rho})\|_{L^2}^2+\|\mathbf{B}\|_{L^4}^4\right]
\Big[\|\sqrt{\rho}\dot{\mathbf{u}}\|_{L^2}^2
+C_0^\frac12\|\nabla\mathbf{B}\|_{L^2}^2\big(1+\|\nabla^2\mathbf{B}\|_{L^2}^2\big)\Big]\mathrm{d}t
\notag\\
&\quad+\frac{CC_0^\frac12}{2\mu+\lambda}
\int_0^T\big[\|\mathbf{B}\|_{L^4}^4\|\nabla\mathbf{B}\|_{L^2}^2
\big(1+\|\nabla^2\mathbf{B}\|_{L^2}^2\big)\big]\mathrm{d}t
+C(1+C_0)\notag\\
&\leq (2+M)^{\exp\big\{3D_2(1+C_0)^6\big\}},
\end{align*}
as the desired \eqref{3.29}.
\end{proof}

Motivated by \cite{Hoff95,Hoff95*,HWZ}, we establish the time-weighted estimate for $\|\sqrt{\rho}{\dot{\mathbf{u}}}\|_{L^2}^2$ and $\|\nabla^2 \mathbf{B}\|_{L^2}^2$.
\begin{lemma}\label{l3.4}
Let \eqref{3.1} be satisfied, then it holds that
\begin{align}\label{3.32}
&\sup_{0\leq t\leq T}\big[\sigma\big(\|\sqrt{\rho}{\dot{\mathbf{u}}}\|_{L^2}^2+\|\nabla^2 \mathbf{B}\|_{L^2}^2+\|\mathbf{B}_t\|_{L^2}^2\big)\big]
+\int_0^T\sigma\big[(\mu+\lambda)\|\divf \mathbf{u}\|_{L^2}^2
+\mu\|\nabla \dot{\mathbf{u}}\|_{L^2}^2+\nu\|\nabla\mathbf{B}_t\|_{L^2}^2\big]\mathrm{d}t\notag\\
&
\leq \exp\bigg\{(2+M)^{\exp\big\{4D_2(1+C_0)^6\big\}}\bigg\}
\end{align}
provided that $\lambda$ satisfies \eqref{lam} with $D\geq 3D_2$,
where $\divf\mathbf{u}\triangleq\divv\mathbf{u}_t+\mathbf{u}\cdot\nabla\divv\mathbf{u}$.
\end{lemma}
\begin{proof}
Operating $\sigma\dot{u}^j[\partial/\partial t+\divv({\mathbf{u}}\cdot)]$ on $\eqref{3.7}^j$, summing all the equalities with respect to $j$, and integrating the resultant over $\mathbb{R}^2$, we get from
$\eqref{a1}_1$ that
\begin{align}\label{3.33}
&\frac{1}{2}\frac{\mathrm{d}}{\mathrm{d}t}\int\sigma\rho|\dot{\mathbf{u}}|^2\mathrm{d}\mathbf{x}
-\frac{\sigma'}{2}\int\rho|\dot{\mathbf{u}}|^2\mathrm{d}\mathbf{x}\notag\\
&=-\sigma\int\dot{u}^j\big[\partial_j P_{t}+\divv(\mathbf{u}\partial_{j}P)\big]\mathrm{d}\mathbf{x}
+\mu\sigma\int\dot{u}^j\big[\Delta u_{t}^j+\divv\big( \mathbf{u}\Delta u^{j}\big)\big]\mathrm{d}\mathbf{x}\notag\\
&\quad+(\mu+\lambda)\sigma\int\dot{u}^j\big[\partial_j\divv\mathbf{u}_t +\divv\big(\mathbf{u}\partial_j\divv\mathbf{u}\big)\big]\mathrm{d}\mathbf{x}
-\sigma\int\dot{u}^j\big[\partial_j(B^iB^i_t)+\divv\big(B^i\partial_jB^i\mathbf{u}\big)\big]
\mathrm{d}\mathbf{x}\notag\\
&\quad+\sigma\int\dot{u}^j\big[\partial_t\big(B^i\partial_iB^j\big)
+\divv\big(B^i\partial_iB^j\mathbf{u}\big)\big]
\mathrm{d}\mathbf{x}
\triangleq \sum_{i=1}^{5}\mathcal{J}_i.
\end{align}

Next, one needs to estimate each $\mathcal{J}_i$.
Integration by parts and applying Cauchy--Schwarz inequality, it follows that
\begin{align}\label{3.34}
\mathcal{J}_{1}&=\sigma\int P_{t}\divv\dot{\mathbf{u}}\mathrm{d}\mathbf{x}-\sigma\int\dot{\mathbf{u}}\cdot\nabla\divv(P\mathbf{u})\mathrm{d}\mathbf{x}
+\sigma\int\dot{u}^j\divv(P\partial_{j}\mathbf{u})\mathrm{d}\mathbf{x}\notag\\
&=\sigma\int \big(P_{t}+\divv(P\mathbf{u})\big)\divv\dot{\mathbf{u}}\mathrm{d}\mathbf{x}
+\sigma\int\dot{\mathbf{u}}\cdot\nabla\mathbf{u}\cdot\nabla P\mathrm{d}\mathbf{x}
+\sigma\int P\dot{\mathbf{u}}\cdot\nabla\divv\mathbf{u}\mathrm{d}\mathbf{x}\notag\\
&=-\sigma\int(\gamma-1)P\divv\mathbf{u}\divv\dot{\mathbf{u}}\mathrm{d}\mathbf{x}
-\sigma\int P\dot{u}_j^iu_i^j \mathrm{d}\mathbf{x}
\notag\\
&\leq \frac{\mu\sigma}{16}\|\nabla\dot{\mathbf{u}}\|_{L^2}^2+C\sigma\|\nabla\mathbf{u}\|_{L^2}^2.
\end{align}
According to Lemma $\ref{E0}$, one has that
\begin{align}\label{3.35}
\mathcal{J}_2&=\mu\sigma\int\dot{u}^j\big[\Delta \dot{u}^j-\Delta(\mathbf{u}\cdot\nabla u^j)+\divv\big( \mathbf{u}\Delta u^{j}\big)\big]\mathrm{d}\mathbf{x}\notag\\
&=\mu\sigma\int\big[-|\nabla\dot{\mathbf{u}}|^2+\dot{u}^j_i(u^ku_k^j)_i
-\dot{u}^j_i(u^ku^j_i)_k-\dot{u}^j(u^k_iu^j_i)_k\big]\mathrm{d}\mathbf{x}\notag\\
&=\mu\sigma\int\big[-|\nabla\dot{\mathbf{u}}|^2+\dot{u}^j_i(u^ku_k^j)_i
-\dot{u}^j_i(u^ku^j_i)_k+\dot{u}_k^j(u^k_iu^j_i)\big]\mathrm{d}\mathbf{x}\notag\\
&\leq-\frac{3\mu\sigma}{4}\|\nabla\dot{\mathbf{u}}\|_{L^2}^2+C\sigma\|\nabla\mathbf{u}\|_{L^4}^4\notag\\
&\leq-\frac{3\mu\sigma}{4}\|\nabla\dot{\mathbf{u}}\|_{L^2}^2
+C\sigma(1+\|\nabla\mathbf{u}\|_{L^2}^2)
(\|\sqrt{\rho}\dot{\mathbf{u}}\|_{L^2}^2+\|\nabla\mathbf{u}\|_{L^2}^2)
+C\sigma(1+C_0)^2\|\nabla\mathbf{B}\|_{L^2}^4\|\nabla^2\mathbf{B}\|_{L^2}^2\notag\\
&\quad+
\frac{C\sigma(1+C_0)^2}{(2\mu+\lambda)^4}\|\nabla\mathbf{B}\|_{L^2}^2\|\sqrt{\rho}\dot{\mathbf{u}}\|_{L^2}^2
+\frac{C}{(2\mu+\lambda)^4}\|P-P(\tilde{\rho})\|_{L^4}^4,
\end{align}
where we have used the fact $0\leq\sigma, \sigma'\leq1$ for $t>0$.
With regard to $\mathcal{J}_3$, we begin by the following decomposition:
\begin{align}\label{3.36}
\mathcal{J}_3&=(\mu+\lambda)\sigma\int\dot{u}^j[\partial_j\divv\mathbf{u}_t+ \divv(\mathbf{u}\partial_j\divv\mathbf{u})]\mathrm{d}\mathbf{x}\notag\\
&=(\mu+\lambda)\sigma\int\dot{u}^j[\partial_j\divv\mathbf{u}_t+ \partial_j\divv(\mathbf{u}\divv\mathbf{u})-\divv(\partial_j\mathbf{u}\divv\mathbf{u})]\mathrm{d}\mathbf{x}\notag\\
&=-(\mu+\lambda)\sigma\int\divv\dot{\mathbf{u}}[\divv\mathbf{u}_t+\divv(\mathbf{u}\divv\mathbf{u})]\mathrm{d}\mathbf{x}-(\mu+\lambda)\sigma\int\dot{u}^j\divv(\partial_j\mathbf{u}\divv\mathbf{u})\mathrm{d}\mathbf{x}\notag\\
&=-(\mu+\lambda)\sigma\int(\divv\mathbf{u}_t+\mathbf{u}\cdot\nabla\divv\mathbf{u}+u_j^iu_i^j)
[\divv\mathbf{u}_t+\mathbf{u}\cdot\nabla\divv\mathbf{u}+(\divv\mathbf{u})^2]\mathrm{d}\mathbf{x}
\notag\\&\quad
-(\mu+\lambda)\sigma\int\dot{u}^j\divv(\partial_j\mathbf{u}\divv\mathbf{u})\mathrm{d}\mathbf{x}\notag\\
&=-(\mu+\lambda)\sigma\int\big[(\divf\mathbf{u})^2
+\divf\mathbf{u}(\divv\mathbf{u})^2+u_j^iu_i^j(\divv\mathbf{u})^2
-\partial_j\mathbf{u}\cdot\nabla\dot{u}^j\divv\mathbf{u}+u_j^iu_i^j\divf\mathbf{u}\big]\mathrm{d}\mathbf{x}\notag\\
&\triangleq-(\mu+\lambda)\sigma\|\divf\mathbf{u}\|_{L^2}^2+\sum_{i=1}^4\mathcal{J}_{3i}.
\end{align}

Our next goal is to bound each $\mathcal{J}_{3i}$. It follows from Gagliardo--Nirenberg inequality, H\"older's inequality, \eqref{1.6}, Lemma $\ref{E0}$, and Lemma $\ref{l3.1}$ that
\begin{align}\label{3.37}
\mathcal{J}_{31}&=-\frac{(\mu+\lambda)\sigma}{(2\mu+\lambda)^2}
\int\divf\mathbf{u}\Big(F+P-P(\tilde{\rho})+\frac12|\mathbf{B}|^2\Big)^2\mathrm{d}\mathbf{x}\notag\\
&\leq\frac{C(\mu+\lambda)\sigma}{(2\mu+\lambda)^2}\|\divf\mathbf{u}\|_{L^2}
\big(\|F\|_{L^4}^2+\|P-P(\tilde{\rho})\|_{L^4}^2+\||\mathbf{B}|^2\|_{L^4}^2\big)
\notag\\
&\leq\frac{C(\mu+\lambda)\sigma}{(2\mu+\lambda)^2}\|\divf\mathbf{u}\|_{L^2}
\big(\|F\|_{L^2}\|\nabla F\|_{L^2}+\|P-P(\tilde{\rho})\|_{L^4}^2
+\|\mathbf{B}\|_{L^\infty}^2\|\mathbf{B}\|_{L^4}^2\big)\notag\\
&\leq\frac{C(\mu+\lambda)\sigma}{(2\mu+\lambda)^2}\|\divf\mathbf{u}\|_{L^2}
\Big[\big(2\mu+\lambda)
\|\divv\mathbf{u}\|_{L^2}+\|P-P(\tilde{\rho})\|_{L^2}+\|\mathbf{B}\|_{L^4}^2\big)
(\|\sqrt{\rho}\dot{\mathbf{u}}\|_{L^2}+\|\mathbf{B}\cdot\nabla\mathbf{B}\|_{L^2})\notag\\
&\quad+\|P-P(\tilde{\rho})\|_{L^4}^2+C_0^\frac34\|\nabla\mathbf{B}\|_{L^2}^2
\|\nabla^2\mathbf{B}\|_{L^2}^\frac12\Big]\notag\\
&\leq\frac{(\mu+\lambda)\sigma}{16}\|\divf\mathbf{u}\|_{L^2}^2
+C\sigma\big(C_0+\|\nabla\mathbf{u}\|_{L^2}^2+C_0\|\nabla\mathbf{B}\|_{L^2}^2\big)
\big(\|\sqrt{\rho}\dot{\mathbf{u}}\|_{L^2}^2+\|\nabla\mathbf{u}\|_{L^2}^2\big)\notag\\
&\quad+C\sigma(1+C_0)^2\|\nabla\mathbf{B}\|_{L^2}^4\|\nabla^2\mathbf{B}\|_{L^2}^2
+\frac{C}{(2\mu+\lambda)^3}\|P-P(\tilde{\rho})\|_{L^4}^4.
\end{align}
Using the Hodge-type decomposition, Lemma $\ref{E0}$, and H\"older's inequality, we obtain that
\begin{align}\label{3.38}
\mathcal{J}_{32}&=-(\mu+\lambda)\sigma\int u_j^iu_i^j(\divv\mathbf{u})^2\mathrm{d}\mathbf{x}\notag\\
&\leq C(\mu+\lambda)\sigma\int\big(|\nabla\mathcal{P}\mathbf{u}|^2+|\nabla\mathcal{Q}\mathbf{u}|^2\big)(\divv\mathbf{u})^2\mathrm{d}\mathbf{x}\notag\\
&\leq C(\mu+\lambda)^2\sigma\|\divv\mathbf{u}\|_{L^4}^4+C\sigma\|\nabla\mathcal{P}\mathbf{u}\|_{L^4}^4\notag\\
&\leq\frac{C(\mu+\lambda)^2\sigma}{(2\mu+\lambda)^4}
\big(\|F\|_{L^4}^4+\|P-P(\tilde{\rho})\|_{L^4}^4
+\||\mathbf{B}|^2\|_{L^4}^4\big)
+C\sigma\big(\|\sqrt{\rho}\dot{\mathbf{u}}\|_{L^{2}}^{2}\|\nabla\mathbf{u}\|_{L^{2}}^{2}
+\|\mathbf{B}\cdot\nabla\mathbf{B}\|_{L^2}^2\|\nabla\mathbf{u}\|_{L^{2}}^{2}\big)
\notag\\
&\leq C\sigma\big(C_0+\|\nabla\mathbf{u}\|_{L^2}^2+C_0\|\nabla\mathbf{B}\|_{L^2}^2\big)
\big(\|\sqrt{\rho}\dot{\mathbf{u}}\|_{L^2}^2+\|\nabla\mathbf{u}\|_{L^2}^2\big)
+C\sigma(1+C_0)^3\|\nabla\mathbf{B}\|_{L^2}^4\|\nabla^2\mathbf{B}\|_{L^2}^2\notag\\
&\quad+\frac{C}{(2\mu+\lambda)^2}\|P-P(\tilde{\rho})\|_{L^4}^4.
\end{align}
Moreover, noting that $\divv(\mathcal{P}\mathbf{u})=0$ and
\begin{equation*}
\mathcal{Q}\mathbf{u}=\frac{-(-\Delta)^{-1}\nabla\big[F+P(\rho)-P(\tilde{\rho})
+\frac{1}{2}|\mathbf{B}|^2\big]}{2\mu+\lambda},
\end{equation*}
one thus gets that
\begin{align}\label{3.39}
\mathcal{J}_{33}&=\frac{(\mu+\lambda)\sigma}{2\mu+\lambda}
\int\partial_j\mathbf{u}\cdot\nabla\dot{u}^j\Big(F+P-P(\tilde{\rho})
+\frac12|\mathbf{B}|^2\Big)\mathrm{d}\mathbf{x}\notag\\
&=\frac{(\mu+\lambda)\sigma}{2\mu+\lambda}
\int\big(\mathcal{P}\mathbf{u}+\mathcal{Q}\mathbf{u}\big)_j
\cdot\nabla\dot{u}^jF\mathrm{d}\mathbf{x}
+\frac{(\mu+\lambda)\sigma}{2(2\mu+\lambda)}
\int\partial_j\mathbf{u}\cdot\nabla\dot{u}^j
\big(2P-2P(\tilde{\rho})+|\mathbf{B}|^2\big)\mathrm{d}\mathbf{x}\notag\\
&=-\frac{(\mu+\lambda)\sigma}{2\mu+\lambda}
\int(\mathcal{P}\mathbf{u})_j\cdot\nabla F\dot{u}^j\mathrm{d}\mathbf{x}
+\frac{(\mu+\lambda)\sigma}{2(2\mu+\lambda)}
\int\partial_j\mathbf{u}\cdot\nabla\dot{u}^j
\big(2P-2P(\tilde{\rho})+|\mathbf{B}|^2\big)\mathrm{d}\mathbf{x}\notag\\
&\quad+\frac{(\mu+\lambda)\sigma}{2\mu+\lambda}
\int(\mathcal{Q}\mathbf{u})_j\cdot\nabla\dot{u}^j F\mathrm{d}\mathbf{x}\notag\\
&\leq C\sigma\|\nabla F\|_{L^2}\|\nabla\mathcal{P}\mathbf{u}\|_{L^4}\|\dot{\mathbf{u}}\|_{L^4}
+C\sigma\|\nabla\mathbf{u}\|_{L^2}\|\nabla\dot{\mathbf{u}}\|_{L^2}
\big(\|P-P(\tilde{\rho})\|_{L^\infty}+\|\mathbf{B}\|_{L^\infty}^2\big)\notag\\
&\quad+\frac{C(\mu+\lambda)\sigma}{(2\mu+\lambda)^2}
\|\nabla\dot{\mathbf{u}}\|_{L^2}\big(\|F\|_{L^4}^2
+\|P(\rho)-P(\tilde{\rho})\|_{L^4}^2+\|\mathbf{B}\|_{L^\infty}^2
\|\mathbf{B}\|_{L^4}^2\big)\notag\\
&\leq \frac{\mu\sigma}{16}\|\nabla\dot{\mathbf{u}}\|_{L^2}^2
+C\sigma\big(C_0+\|\nabla\mathbf{u}\|_{L^{2}}^{2}
+C_0\|\nabla\mathbf{B}\|_{L^{2}}^{2}\big)
\big(\|\sqrt{\rho}\dot{\mathbf{u}}\|_{L^{2}}^{2}
+\|\nabla\mathbf{u}\|_{L^{2}}^{2}\big)
\big(1+\|\sqrt{\rho}\dot{\mathbf{u}}\|_{L^{2}}^{2}\big)\notag\\
&\quad+C\sigma(1+C_0)^3\|\nabla\mathbf{B}\|_{L^2}^4
\|\nabla^2\mathbf{B}\|_{L^2}^2+\frac{C}{(2\mu+\lambda)^2}
\|P-P(\tilde{\rho})\|_{L^4}^4,
\end{align}
where we have used the estimate below, following by replacing $\mathbf{u}$ with $\dot{\mathbf{u}}$ in the arguments for \eqref{3.11},
\begin{equation*}
\|\dot{\mathbf{u}}\|_{L^2}^2
\leq C\|\sqrt{\rho}\dot{\mathbf{u}}\|_{L^2}^2+CC_0\|\nabla\dot{\mathbf{u}}\|_{L^2}^2.
\end{equation*}
As for $\mathcal{J}_{34}$, the Hodge-type decomposition leads to
\begin{align}\label{3.40}
\mathcal{J}_{34}
&=-(\mu+\lambda)\sigma\int\divf\mathbf{u}(\mathcal{Q}\mathbf{u}+\mathcal{P}\mathbf{u})_j^i(\mathcal{Q}\mathbf{u}
+\mathcal{P}\mathbf{u})_i^j\mathrm{d}\mathbf{x}\notag\\
&=-(\mu+\lambda)\sigma\int\divf\mathbf{u}\big[(\mathcal{Q}\mathbf{u})_j^i(\mathcal{Q}\mathbf{u})_i^j+
(\mathcal{Q}\mathbf{u})_j^i(\mathcal{P}\mathbf{u})_i^j+(\mathcal{P}\mathbf{u})_j^i(\mathcal{Q}\mathbf{u})_i^j
+(\mathcal{P}\mathbf{u})_j^i(\mathcal{P}\mathbf{u})_i^j\big]\mathrm{d}\mathbf{x}\notag\\
&\leq-(\mu+\lambda)\sigma\int\divf\mathbf{u}(\mathcal{P}\mathbf{u})_j^i(\mathcal{P}\mathbf{u})_i^j\mathrm{d}\mathbf{x}
+\frac{(\mu+\lambda)\sigma}{16}\|\divf\mathbf{u}\|_{L^2}^2+C(\mu+\lambda)\sigma\int|\nabla\mathbf{u}|^2
|\nabla\mathcal{Q}\mathbf{u}|^2\mathrm{d}\mathbf{x}\notag\\
&\leq-(\mu+\lambda)\sigma\int\divf\mathbf{u}(\mathcal{P}\mathbf{u})_j^i(\mathcal{P}\mathbf{u})_i^j\mathrm{d}\mathbf{x}
+\frac{(\mu+\lambda)\sigma}{16}\|\divf\mathbf{u}\|_{L^2}^2+C\sigma\|\nabla\mathbf{u}\|_{L^4}^4
+C(\mu+\lambda)^2\sigma\|\divv\mathbf{u}\|_{L^4}^4\notag\\
&\leq-\frac{(\mu+\lambda)\sigma}{2\mu+\lambda}\int F_t(\mathcal{P}\mathbf{u})_j^i(\mathcal{P}\mathbf{u})_i^j\mathrm{d}\mathbf{x}
+\frac{(\mu+\lambda)\sigma}{16}\|\divf\mathbf{u}\|_{L^2}^2
+C\sigma(1+C_0)^3\|\nabla\mathbf{B}\|_{L^2}^4
\big(\|\nabla^2\mathbf{B}\|_{L^2}^2+\|\mathbf{B}_t\|_{L^2}^2\big)\notag\\
&\quad
+C\sigma\big(1+C_0+\|\nabla\mathbf{u}\|_{L^{2}}^{2}
+C_0\|\nabla\mathbf{B}\|_{L^{2}}^{2}\big)
\big(\|\sqrt{\rho}\dot{\mathbf{u}}\|_{L^{2}}^{2}
+\|\nabla\mathbf{u}\|_{L^{2}}^{2}+\|\nabla\mathbf{B}\|_{L^{2}}^{2}\big)
\big(1+\|\sqrt{\rho}\dot{\mathbf{u}}\|_{L^{2}}^{2}\big)\notag\\
&\quad+\frac{C}{(2\mu+\lambda)^2}\|P-P(\tilde{\rho})\|_{L^4}^4,
\end{align}
where in the last inequality we have employed
\begin{align*}
&-(\mu+\lambda)\sigma\int(\divv\mathbf{u}_t+\mathbf{u}\cdot\nabla\divv\mathbf{u})(\mathcal{P}\mathbf{u})_j^i(\mathcal{P}\mathbf{u})_i^j\mathrm{d}\mathbf{x}\notag\\
&=-\frac{(\mu+\lambda)\sigma}{2\mu+\lambda}\int\big[(P-P(\tilde{\rho}))_t+\mathbf{u}\cdot\nabla(P-P(\tilde{\rho}))+\mathbf{B}\cdot\mathbf{B}_t\big](\mathcal{P}\mathbf{u})_j^i(\mathcal{P}\mathbf{u})_i^j\mathrm{d}\mathbf{x}
\notag\\
&\quad-\frac{(\mu+\lambda)\sigma}{2\mu+\lambda}\int\mathbf{u}\cdot\nabla\Big(F+\frac12|\mathbf{B}|^2\Big)(\mathcal{P}\mathbf{u})_{j}^{i}(\mathcal{P}\mathbf{u})_{i}^{j}\mathrm{d}\mathbf{x}
-\frac{(\mu+\lambda)\sigma}{2\mu+\lambda}\int F_t(\mathcal{P}\mathbf{u})_j^i(\mathcal{P}\mathbf{u})_i^j\mathrm{d}\mathbf{x}\notag\\
&=\frac{(\mu+\lambda)\sigma}{2\mu+\lambda}\int(\gamma P\divv\mathbf{u}-\mathbf{B}\cdot\mathbf{B}_t)(\mathcal{P}\mathbf{u})_j^i(\mathcal{P}\mathbf{u})_i^j\mathrm{d}\mathbf{x}
-\frac{(\mu+\lambda)\sigma}{2\mu+\lambda}\int\mathbf{u}\cdot\nabla\Big(F+\frac12|\mathbf{B}|^2\Big)(\mathcal{P}\mathbf{u})_{j}^{i}(\mathcal{P}\mathbf{u})_{i}^{j}\mathrm{d}\mathbf{x}
\notag\\
&\quad
-\frac{(\mu+\lambda)\sigma}{2\mu+\lambda}\int F_t(\mathcal{P}\mathbf{u})_j^i(\mathcal{P}\mathbf{u})_i^j\mathrm{d}\mathbf{x}\notag\\
&\leq C\sigma\big(\|\nabla\mathcal{P}\mathbf{u}\|_{L^4}^2\|\nabla\mathbf{u}\|_{L^2}
+\|\nabla\mathcal{P}\mathbf{u}\|_{L^8}^2\|\mathbf{B}\|_{L^4}\|\mathbf{B}_t\|_{L^2}\big)
+C\sigma\|\mathbf{u}\|_{L^4}\|\nabla\mathcal{P}\mathbf{u}\|_{L^8}^2\big(\|\nabla F\|_{L^2}+\|\mathbf{B}\|_{L^\infty}\|\nabla \mathbf{B}\|_{L^2}\big)
\notag\\
&\quad
-\frac{(\mu+\lambda)\sigma}{2\mu+\lambda}\int F_t(\mathcal{P}\mathbf{u})_j^i(\mathcal{P}\mathbf{u})_i^j\mathrm{d}\mathbf{x}\notag\\
&\leq
C\sigma\big(1+C_0+\|\nabla\mathbf{u}\|_{L^{2}}^{2}
+C_0\|\nabla\mathbf{B}\|_{L^{2}}^{2}\big)
\big(\|\sqrt{\rho}\dot{\mathbf{u}}\|_{L^{2}}^{2}
+\|\nabla\mathbf{u}\|_{L^{2}}^{2}+\|\nabla\mathbf{B}\|_{L^{2}}^{2}\big)
\big(1+\|\sqrt{\rho}\dot{\mathbf{u}}\|_{L^{2}}^{2}\big)\notag\\
&\quad+C\sigma(1+C_0)^2\|\nabla\mathbf{B}\|_{L^2}^4
\big(\|\nabla^2\mathbf{B}\|_{L^2}^2+\|\mathbf{B}_t\|_{L^2}^2\big)
-\frac{(\mu+\lambda)\sigma}{2\mu+\lambda}\int F_t(\mathcal{P}\mathbf{u})_j^i(\mathcal{P}\mathbf{u})_i^j\mathrm{d}\mathbf{x}.
\end{align*}
Due to $\divv(\mathcal{P}\mathbf{u})=0$, the term involving $F_t$ in \eqref{3.40} can be handed via integration by parts as follows
\begin{align}\label{3.41}
&-\frac{(\mu+\lambda)\sigma}{2\mu+\lambda}\int F_t
(\mathcal{P}\mathbf{u})_{j}^{i}(\mathcal{P}\mathbf{u})_{i}^{j}\mathrm{d}\mathbf{x}\notag\\
&=-\frac{\mu+\lambda}{2\mu+\lambda}\frac{\mathrm{d}}{\mathrm{d}t}\int\sigma F
(\mathcal{P}\mathbf{u})_{j}^{i}(\mathcal{P}\mathbf{u})_{i}^{j}\mathrm{d}\mathbf{x}
+\frac{\mu+\lambda}{2\mu+\lambda}\int \sigma'F
(\mathcal{P}\mathbf{u})_{j}^{i}(\mathcal{P}\mathbf{u})_{i}^{j}\mathrm{d}\mathbf{x}
+\frac{\mu+\lambda}{2\mu+\lambda}\int \sigma F
(\mathcal{P}\mathbf{u})_{jt}^{i}(\mathcal{P}\mathbf{u})_{i}^{j}\mathrm{d}\mathbf{x}
\notag\\
&\quad+\frac{\mu+\lambda}{2\mu+\lambda}\int \sigma F
(\mathcal{P}\mathbf{u})_{j}^{i}(\mathcal{P}\mathbf{u})_{it}^{j}\mathrm{d}\mathbf{x}\notag\\
&=\frac{\mu+\lambda}{2\mu+\lambda}\frac{\mathrm{d}}{\mathrm{d}t}\int\sigma F_j
(\mathcal{P}\mathbf{u})^{i}(\mathcal{P}\mathbf{u})_{i}^{j}\mathrm{d}\mathbf{x}
-\frac{\mu+\lambda}{2\mu+\lambda}\int \sigma'F_j
(\mathcal{P}\mathbf{u})^{i}(\mathcal{P}\mathbf{u})_{i}^{j}\mathrm{d}\mathbf{x}
-\frac{\mu+\lambda}{2\mu+\lambda}\int \sigma F_j
(\mathcal{P}\mathbf{u})_{t}^{i}(\mathcal{P}\mathbf{u})_i^{j}\mathrm{d}\mathbf{x}
\notag\\&\quad
-\frac{\mu+\lambda}{2\mu+\lambda}\int \sigma F_i
(\mathcal{P}\mathbf{u})^{i}_j(\mathcal{P}\mathbf{u})_{t}^{j}\mathrm{d}\mathbf{x}
\notag\\
&=\frac{\mu+\lambda}{2\mu+\lambda}\frac{\mathrm{d}}{\mathrm{d}t}\int\sigma F_j
(\mathcal{P}\mathbf{u})^{i}(\mathcal{P}\mathbf{u})_{i}^{j}\mathrm{d}\mathbf{x}
-\frac{\mu+\lambda}{2\mu+\lambda}\int \sigma'F_j
(\mathcal{P}\mathbf{u})^{i}(\mathcal{P}\mathbf{u})_{i}^{j}\mathrm{d}\mathbf{x}
\notag\\&\quad
-\frac{\mu+\lambda}{2\mu+\lambda}\int \sigma F_j
\big(\mathcal{P}(\dot{\mathbf{u}}-\mathbf{u}\cdot\nabla\mathbf{u})\big)^{i}(\mathcal{P}\mathbf{u})_i^{j}\mathrm{d}\mathbf{x}
-\frac{\mu+\lambda}{2\mu+\lambda}\int \sigma F_i
(\mathcal{P}\mathbf{u})_j^{i}\big(\mathcal{P}(\dot{\mathbf{u}}-\mathbf{u}\cdot\nabla\mathbf{u})\big)^{j}\mathrm{d}\mathbf{x}\notag\\
&\leq\frac{\mu+\lambda}{2\mu+\lambda}\frac{\mathrm{d}}{\mathrm{d}t}\int\sigma F_j
(\mathcal{P}\mathbf{u})^{i}(\mathcal{P}\mathbf{u})_{i}^{j}\mathrm{d}\mathbf{x}
+C\|\nabla F\|_{L^2}\|\mathbf{u}\|_{L^4}\|\nabla\mathcal{P}\mathbf{u}\|_{L^4}
+C\sigma\|\nabla F\|_{L^2}\|\dot{\mathbf{u}}\|_{L^4}\|\nabla\mathcal{P}\mathbf{u}\|_{L^4}
\notag\\&\quad
+C\sigma\|\nabla F\|_{L^2}\|\mathbf{u}\|_{L^{8}}\|\nabla\mathbf{u}\|_{L^4}\|\nabla\mathcal{P}\mathbf{u}\|_{L^8}\notag\\
&\leq\frac{\mu+\lambda}{2\mu+\lambda}\frac{\mathrm{d}}{\mathrm{d}t}\int\sigma F_j
(\mathcal{P}\mathbf{u})^{i}(\mathcal{P}\mathbf{u})_{i}^{j}\mathrm{d}\mathbf{x}
+\frac{\mu\sigma}{16}\|\nabla \dot{\mathbf{u}}\|_{L^2}^2
+C\sigma(1+C_0)^2\|\nabla\mathbf{B}\|_{L^2}^4\|\nabla^2\mathbf{B}\|_{L^2}^2\notag\\
&\quad
+C\big(1+C_0+\|\nabla\mathbf{u}\|_{L^2}^2+C_0\|\nabla\mathbf{B}\|_{L^2}^2\big)
\big(1+\sigma\|\sqrt{\rho}\dot{\mathbf{u}}\|_{L^{2}}^{2}
+\|\nabla\mathbf{u}\|_{L^{2}}^{2}\big)
\big(\|\sqrt{\rho}\dot{\mathbf{u}}\|_{L^2}^2
+\|\nabla\mathbf{u}\|_{L^2}^2\big)\notag\\&\quad
+\frac{C}{(2\mu+\lambda)^4}\|P-P(\tilde{\rho})\|_{L^4}^4.
\end{align}
Thus, substituting \eqref{3.37}--\eqref{3.41} into \eqref{3.36}, we obtain that
\begin{align}\label{3.42}
\mathcal{J}_{3}
&\leq\frac{\mu+\lambda}{2\mu+\lambda}\frac{\mathrm{d}}{\mathrm{d}t}\int\sigma F_j
(\mathcal{P}\mathbf{u})^{i}(\mathcal{P}\mathbf{u})_{i}^{j}\mathrm{d}\mathbf{x}
-\frac{3(\mu+\lambda)\sigma}{4}\|\divf\mathbf{u}\|_{L^2}^2
+\frac{\mu\sigma}{8}\|\nabla \dot{\mathbf{u}}\|_{L^2}^2\notag\\
&\quad+C\big(1+C_0+\|\nabla\mathbf{u}\|_{L^2}^2+C_0\|\nabla\mathbf{B}\|_{L^2}^2\big)
\big(1+\sigma\|\sqrt{\rho}\dot{\mathbf{u}}\|_{L^{2}}^{2}
+\|\nabla\mathbf{u}\|_{L^{2}}^{2}\big)
\big(\|\sqrt{\rho}\dot{\mathbf{u}}\|_{L^2}^2
+\|\nabla\mathbf{u}\|_{L^2}^2+\|\nabla\mathbf{B}\|_{L^2}^2\big)
\notag\\
&\quad+C\sigma(1+C_0)^3\|\nabla\mathbf{B}\|_{L^2}^4
\big(\|\nabla^2\mathbf{B}\|_{L^2}^2+\|\mathbf{B}_t\|_{L^2}^2\big)
+\frac{C}{(2\mu+\lambda)^2}\|P-P(\tilde{\rho})\|_{L^4}^4.
\end{align}
For the terms $\mathcal{J}_{4}$ and $\mathcal{J}_{5}$, it follows from Gagliardo--Nirenberg inequality that
\begin{align}\label{3.43}
 &\,\mathcal{J}_{4}+\mathcal{J}_{5}\notag\\
 &=-\sigma\int\dot{u}^j\big[\partial_j(B^iB^i_t)+\divv(B^i\partial_jB^i\mathbf{u})\big]
\mathrm{d}\mathbf{x}
+\sigma\int\dot{u}^j\big[\partial_t\big(B^i\partial_iB^j\big)
+\divv\big(B^i\partial_iB^j\mathbf{u}\big)\big]
\mathrm{d}\mathbf{x}
\notag\\
&=\sigma\int\big[\partial_j\dot{u}^jB^iB^i_t
+\partial_k\dot{u}^jB^i\partial_jB^iu^k
-\partial_i\dot{u}^j\big(B^jB^i_t+B^j_tB^i\big)
-\partial_k\dot{u}^jB^i\partial_iB^ju^k\big]\mathrm{d}\mathbf{x}
\notag\\
&\leq C\sigma\|\nabla\dot{\mathbf{u}}\|_{L^2}
\big(\|\mathbf{B}\|_{L^4}\|\mathbf{B}_t\|_{L^4}
+\|\mathbf{B}\|_{L^8}\|\nabla\mathbf{B}\|_{L^4}\|\mathbf{u}\|_{L^8}\big)\notag\\
&\leq \frac{\mu\sigma}{16}\|\nabla \dot{\mathbf{u}}\|_{L^2}^2
+\frac{\nu\sigma}{8}\|\nabla\mathbf{B}_t\|_{L^2}^2
+C\sigma (1+C_0)\big(1+\|\nabla\mathbf{B}\|_{L^2}^4+\|\nabla\mathbf{u}\|_{L^2}^2\big)
\big(\|\nabla^2\mathbf{B}\|_{L^2}^2+\|\mathbf{B}_t\|_{L^2}^2\big).
\end{align}
Consequently, inserting \eqref{3.34}, \eqref{3.35}, \eqref{3.42}, and \eqref{3.43} into \eqref{3.33}, one derives from \eqref{3.6} that
\begin{align}\label{3.44}
&\frac{\mathrm{d}}{\mathrm{d}t}\Big(\frac{\sigma}{2}\|\sqrt{\rho}\dot{\mathbf{u}}\|_{L^{2}}^{2}
-\frac{\mu+\lambda}{2\mu+\lambda}\int\sigma F_j
(\mathcal{P}\mathbf{u})^{i}(\mathcal{P}\mathbf{u})_{i}^{j}\mathrm{d}\mathbf{x}\Big)
+\frac{\mu\sigma}{2}\|\nabla\dot{\mathbf{u}}\|_{L^2}^2
+\frac{(\mu+\lambda)\sigma}{2}\|\divf\mathbf{u}\|_{L^2}^2\notag\\
&\leq(2+M)^{\exp\big\{\frac{5}{2}D_2(1+C_0)^6\big\}}
\big(1+\sigma\|\sqrt{\rho}\dot{\mathbf{u}}\|_{L^{2}}^{2}
+\|\nabla\mathbf{u}\|_{L^{2}}^{2}\big)
\big(\|\sqrt{\rho}\dot{\mathbf{u}}\|_{L^{2}}^{2}
+\|\nabla\mathbf{u}\|_{L^{2}}^{2}+\|\nabla\mathbf{B}\|_{L^{2}}^{2}\big)\notag\\
&\quad
+(2+M)^{\exp\big\{\frac{5}{2}D_2(1+C_0)^6\big\}}\big(\sigma\|\nabla^2\mathbf{B}\|_{L^2}^2
+\sigma\|\mathbf{B}_t\|_{L^2}^2\big)
+\frac{\nu\sigma}{8}\|\nabla\mathbf{B}_t\|_{L^2}^2+\frac{C}{(2\mu+\lambda)^2}\|P-P(\tilde{\rho})\|_{L^4}^4.
\end{align}

It remains to estimate $\sigma\|\mathbf{B}_t\|_{L^2}^2$ and $\sigma\|\nabla^2\mathbf{B}\|_{L^2}^2$. To this end, differentiating $\eqref{a1}_3$ with respect to $t$ and multiplying the resultant by $\sigma\mathbf{B}_t$, we obtain from integration by parts that
\begin{align}\label{3.45}
 &\,\frac12\frac{\mathrm{d}}{\mathrm{d}t}\big(\sigma\|\mathbf{B}_t\|_{L^{2}}^2\big)
 +\nu\sigma\|\nabla\mathbf{B}_t\|_{L^{2}}^2
 -\frac12\sigma'\|\mathbf{B}_t\|_{L^{2}}^2\notag\\
 &=\sigma\int\big(\mathbf{B}_t\cdot\nabla\mathbf{u}-\mathbf{u}\cdot\nabla\mathbf{B}_t-\mathbf{B}_t\divv\mathbf{u}\big)\cdot\mathbf{B}_t\mathrm{d}\mathbf{x}
 +\sigma\int\big(\mathbf{B}\cdot\nabla\dot{\mathbf{u}}-\dot{\mathbf{u}}\cdot\nabla\mathbf{B}
 -\mathbf{B}\divv\dot{\mathbf{u}}\big)\cdot\mathbf{B}_t\mathrm{d}\mathbf{x}
 \notag\\&\quad
 -\sigma\int\big(\mathbf{B}\cdot\nabla(\mathbf{u}\cdot\nabla\mathbf{u})
 -(\mathbf{u}\cdot\nabla\mathbf{u})\cdot\nabla\mathbf{B}-\mathbf{B}\divv(\mathbf{u}\cdot\nabla\mathbf{u})\big)
 \cdot\mathbf{B}_t\mathrm{d}\mathbf{x}
 \triangleq \mathcal{L}_1+\mathcal{L}_2+\mathcal{L}_3.
\end{align}
By similar arguments, we conclude that
\begin{align}\label{3.46}
\mathcal{L}_1&\leq C\sigma\|\mathbf{B}_t\|_{L^4}^2\|\nabla\mathbf{u}\|_{L^2}\leq
\frac{\nu\sigma}{8}\|\nabla\mathbf{B}_t\|_{L^2}^2
+C\sigma\|\nabla\mathbf{u}\|_{L^2}^2\|\mathbf{B}_t\|_{L^2}^2,\\
\mathcal{L}_2&\leq C\sigma
\|\mathbf{B}_t\|_{L^4}\big(\|\nabla\dot{\mathbf{u}}\|_{L^2}\|\mathbf{B}\|_{L^4}
+\|\dot{\mathbf{u}}\|_{L^4}\|\nabla\mathbf{B}\|_{L^2}\big)\notag\\
&\leq \frac{\mu\sigma}{16}\|\nabla\dot{\mathbf{u}}\|_{L^2}^2
+\frac{\nu\sigma}{8}\|\nabla\mathbf{B}_t\|_{L^2}^2
+C\sigma C_0\big(1+\|\nabla\mathbf{B}\|_{L^2}^4\big)
\|\mathbf{B}_t\|_{L^2}^2+C\sigma\|\sqrt{\rho}\dot{\mathbf{u}}\|_{L^2}^2,\label{3.47}\\
\mathcal{L}_3
&\leq C\sigma\big(\|\mathbf{B}\|_{L^4}\|\mathbf{B}_t\|_{L^4}\|\nabla\mathbf{u}\|_{L^4}^2
+\|\mathbf{u}\|_{L^\infty}\|\nabla\mathbf{u}\|_{L^4}\|\mathbf{B}\|_{L^4}\|\nabla\mathbf{B}_t\|_{L^2}\big)\notag\\
&\leq\frac{\nu\sigma}{8}\|\nabla\mathbf{B}_t\|_{L^2}^2
+C\sigma(1+C_0)^3\big(1+\|\nabla\mathbf{B}\|_{L^2}^4\big)
\big(1+\|\nabla\mathbf{u}\|_{L^2}^2\big)
\big(\|\sqrt{\rho}\dot{\mathbf{u}}\|_{L^2}^2+\|\nabla\mathbf{u}\|_{L^2}^2\big)
\notag\\
&\quad
+C\sigma (1+C_0)^2\big(1+\|\nabla\mathbf{B}\|_{L^2}^4\big)
\big(\|\nabla^2\mathbf{B}\|_{L^2}^2+\|\mathbf{B}_t\|_{L^2}^2\big)
+\frac{C}{(2\mu+\lambda)^4}\|P-P(\tilde{\rho})\|_{L^4}^4.\label{3.48}
\end{align}
Moreover, using the $L^2$-estimate of elliptic system, one gets from $\eqref{a1}_3$ that
\begin{align}\label{3.49}
  \|\nabla^2\mathbf{B}\|_{L^2}^2&\leq C(\nu)\big(\|\mathbf{B}_t\|_{L^2}^2+\||\mathbf{u}||\nabla\mathbf{B}|\|_{L^2}^2+\||\mathbf{B}||\nabla\mathbf{u}|\|_{L^2}^2\big)\notag\\
&\leq C\big(\|\mathbf{B}_t\|_{L^2}^2+\|\mathbf{u}\|_{L^4}^2\|\nabla\mathbf{B}\|_{L^4}^2
+\|\mathbf{B}\|_{L^\infty}^2\|\nabla\mathbf{u}\|_{L^2}^2\big)\notag\\
&\leq \frac{1}{2}\|\nabla^2\mathbf{B}\|_{L^2}^2
+C\|\mathbf{B}_t\|_{L^2}^2
+C(1+C_0)\|\nabla\mathbf{u}\|_{L^2}^2
\big(1+\|\nabla\mathbf{u}\|_{L^2}^2\big)\big(1+\|\nabla\mathbf{B}\|_{L^2}^2\big).
\end{align}
Substituting \eqref{3.46}--\eqref{3.48} into \eqref{3.45}, and then adding \eqref{3.44},
it follows from \eqref{3.49} that
\begin{align}\label{3.50}
&\frac{\mathrm{d}}{\mathrm{d}t}\bigg(\frac{\sigma}{2}
\|\sqrt{\rho}\dot{\mathbf{u}}\|_{L^{2}}^{2}
+\frac{\sigma}{2}\|\mathbf{B}_t\|_{L^{2}}^2
-\frac{\mu+\lambda}{2\mu+\lambda}\int\sigma F_j
(\mathcal{P}\mathbf{u})^{i}(\mathcal{P}\mathbf{u})_{i}^{j}\mathrm{d}\mathbf{x}\bigg)
\notag\\&\quad
+\frac{\mu\sigma}{4}\|\nabla\dot{\mathbf{u}}\|_{L^2}^2
+\frac{(\mu+\lambda)\sigma}{4}\|\divf\mathbf{u}\|_{L^2}^2
+\frac{\nu\sigma}{4}\|\nabla\mathbf{B}_t\|_{L^{2}}^2\notag\\
&\leq(2+M)^{\exp\big\{3D_2(1+C_0)^6\big\}}
\big(1+\sigma\|\sqrt{\rho}\dot{\mathbf{u}}\|_{L^{2}}^{2}
+\|\nabla\mathbf{u}\|_{L^{2}}^{2}\big)
\big(\|\sqrt{\rho}\dot{\mathbf{u}}\|_{L^{2}}^{2}
+\|\nabla\mathbf{u}\|_{L^{2}}^{2}+\|\nabla\mathbf{B}\|_{L^{2}}^{2}\big)\notag\\
&\quad
+(2+M)^{\exp\big\{3D_2(1+C_0)^6\big\}}\bigg[ \big(1+\sigma\|\mathbf{B}_t\|_{L^2}^2\big)\|\mathbf{B}_t\|_{L^2}^2
+\frac{C}{(2\mu+\lambda)^2}\|P-P(\tilde{\rho})\|_{L^4}^4\bigg],
\end{align}
where we observe from Lemma $\ref{E0}$ and \eqref{3.49} that
\begin{align*}
&\frac{\mu+\lambda}{2\mu+\lambda}\left|\int\sigma F_j
(\mathcal{P}\mathbf{u})^{i}(\mathcal{P}\mathbf{u})_{i}^{j}\mathrm{d}\mathbf{x}\right|
\leq C\sigma\|\nabla F\|_{L^2}\|\mathbf{u}\|_{L^4}\|\nabla\mathcal{P}\mathbf{u}\|_{L^4}\notag\\
&\leq\frac{\sigma}{4}\|\sqrt{\rho}\dot{\mathbf{u}}\|_{L^{2}}^2
+\frac{\sigma}{4}\|\mathbf{B}_t\|_{L^{2}}^2
+C\sigma(1+C_0)
\big(\|\nabla\mathbf{u}\|_{L^{2}}^2+\|\nabla\mathbf{B}\|_{L^{2}}^2\big)^2
\big(1+\|\nabla\mathbf{u}\|_{L^{2}}^2\big).
\end{align*}

Now we define an auxiliary functional $\mathcal{E}_2(t)$ as
\begin{align*}
\mathcal{E}_2(t)\triangleq\frac{\sigma}{2}\|\sqrt{\rho}\dot{\mathbf{u}}\|_{L^{2}}^2
+\frac{\sigma}{2}\|\mathbf{B}_t\|_{L^{2}}^2
-\frac{\mu+\lambda}{2\mu+\lambda}\int\sigma F_j
(\mathcal{P}\mathbf{u})^{i}(\mathcal{P}\mathbf{u})_{i}^{j}\mathrm{d}\mathbf{x}
+(2+M)^{\exp\big\{\frac{5}{2}D_2(1+C_0)^6\big\}}
\mathcal{E}_1(t),
\end{align*}
where, by the definition of $\mathcal{E}_1(t)$ in \eqref{3.23}, one can see that
\begin{equation}\label{3.51}
\mathcal{E}_2(t)\thicksim \sigma\|\sqrt{\rho}\dot{\mathbf{u}}\|_{L^{2}}^2
+\sigma\|\mathbf{B}_t\|_{L^{2}}^2+\mathcal{E}_1(t).
\end{equation}
Setting
\begin{equation*}
\begin{cases}
f_2(t)\triangleq2+\mathcal{E}_2(t),\\
g_2(t)\triangleq(2+M)^{\exp\big\{\frac{13}{4}D_2(1+C_0)^6\big\}}
\bigg(\|\sqrt{\rho}\dot{\mathbf{u}}\|_{L^{2}}^2+\|\nabla\mathbf{u}\|_{L^{2}}^2
+\|\mathbf{B}_t\|_{L^{2}}^2+\|\nabla\mathbf{B}\|_{L^{2}}^2
+\frac{\|P-P(\tilde{\rho})\|_{L^4}^4}{(2\mu+\lambda)^2}\bigg)
\end{cases}
\end{equation*}
and taking the summation
\begin{align*}
(2+M)^{\exp\big\{\frac{7}{2}D_2(1+C_0)^6\big\}}
\times\big(\eqref{3.5}+\eqref{3.22}\big)+\eqref{3.50},
\end{align*}
one gets from \eqref{3.51} that
\begin{align*}
f'_2(t)\leq g_2(t)f_2(t)
\end{align*}
provided that $\lambda$ satisfies \eqref{lam} with $D\geq 3D_2$.
Thus, it follows from Gronwall's inequality and Lemmas \ref{l3.1}\text{--}\ref{l3.3} that
\begin{equation}\label{3.52}
  \sup_{0\leq t\leq T}\big(\sigma\|\sqrt{\rho}\dot{\mathbf{u}}\|_{L^{2}}^2
  +\sigma\|\mathbf{B}_t\|_{L^{2}}^2\big)\leq
  \exp\bigg\{(2+M)^{\exp\big\{\frac{13}{4}D_2(1+C_0)^6\big\}}\bigg\}.
\end{equation}
Moreover, \eqref{3.49} implies that
\begin{equation}\label{3.53}
  \sup_{0\leq t\leq T}\big(\sigma\|\nabla^2\mathbf{B}\|_{L^{2}}^2\big)\leq
  \exp\bigg\{(2+M)^{\exp\big\{\frac{7}{2}D_2(1+C_0)^6\big\}}\bigg\}.
\end{equation}
Integrating \eqref{3.50} with respect to $t$ over $(0,T)$, one infers from \eqref{3.52} and Lemmas $\ref{l3.1}\text{--}\ref{l3.3}$ that
\begin{align*}
&\int_0^T\big[\mu\sigma\|\nabla\dot{\mathbf{u}}\|_{L^2}^2
+(\mu+\lambda)\sigma\|\divf\mathbf{u}\|_{L^2}^2
+\nu\sigma\|\nabla\mathbf{B}_t\|_{L^{2}}^2\big]\mathrm{d}t\notag\\
&\leq C(1+M)+C(1+C_0)(2+M)^{\exp\big\{3D_2(1+C_0)^6\big\}}
 \exp\bigg\{(2+M)^{\exp\big\{\frac{7}{2}D_2(1+C_0)^6\big\}}\bigg\}\notag\\
&\leq
 \exp\bigg\{(2+M)^{\exp\big\{\frac{15}{4}D_2(1+C_0)^6\big\}}\bigg\},
\end{align*}
which along with \eqref{3.52} and \eqref{3.53} leads to \eqref{3.32}.
\end{proof}

Finally, inspired by \cite{DE97}, we establish the upper bound of density.
\begin{lemma}\label{l3.5}
Under the assumption \eqref{3.1}, it holds that
\begin{align*}
0\leq\rho(\mathbf{x},t)\leq\frac{7}{4}\hat{\rho}~\textit{a.e.}~\mathrm{on}~\mathbb{R}^2\times[0,T]
\end{align*}
provided that $\lambda$ satisfies \eqref{lam} with $D\geq 5D_2$.
\end{lemma}
\begin{proof}
Let $\mathbf{y}\in\mathbb{R}^2$ and define the corresponding particle path $\mathbf{x}(t)$ by
\begin{align*}
\begin{cases}
\mathbf{\dot{x}}(t,\mathbf{y})=\mathbf{u}(\mathbf{x}(t,\mathbf{y}),\mathbf{y)},\\
\mathbf{\dot{x}}(t_0,\mathbf{y})=\mathbf{y}.
\end{cases}
\end{align*}
Assume that there exists $t_1\leq T$ satisfying $\rho(\mathbf{x}(t_1), t_1) = \frac{7}{4}\hat{\rho}$, we take a minimal value of $t_1$ and then choose a maximal value of $t_0<t_1$ such that $\rho(\mathbf{x}(t_0), t_0)=\frac{3}{2}\hat{\rho}$. Thus, $\rho(\mathbf{x}(t),t)\in[\frac{3}{2}\hat{\rho},\frac{7}{4}\hat{\rho}]$ for $t\in[t_0,t_1]$. We divide the argument into two cases.

\textbf{Case 1:} $t_0<t_1\leq1$. The definition of $F$ in \eqref{1.6} and $\eqref{a1}_1$ imply that
\begin{equation*}
(2\mu+\lambda)\frac{\mathrm{d}}{\mathrm{d}t}\ln\rho(\mathbf{x}(t),t)
+P(\rho(\mathbf{x}(t),t))-P(\tilde{\rho})+\frac12|\mathbf{B}(\mathbf{x}(t),t)|^2=-F(\mathbf{x}(t),t),
\end{equation*}
where $\frac{\mathrm{d}\rho}{\mathrm{d}t}\triangleq \rho_t+\mathbf{u}\cdot\nabla\rho$. Integrating the above equality from $t_0$ to $t_1$ and abbreviating $\rho(\mathbf{x},t)$ by $\rho(t)$ for
convenience, one gets that
\begin{equation}\label{3.54}
\ln\rho(\tau)\big|_{t_0}^{t_1}+\frac{1}{2\mu+\lambda}\int_{t_0}^{t_1}\Big[P(\rho(\tau))
-P(\tilde{\rho})+\frac12|\mathbf{B}(\mathbf{x}(\tau),\tau)|^2\Big]\mathrm{d}\tau=-\frac{1}{2\mu+\lambda}\int_{t_0}^{t_1}F(\mathbf{x}(\tau),\tau)\mathrm{d}\tau.
\end{equation}
It follows from Gagliardo--Nirenberg inequality,
Lemmas $\ref{E0}$, $\ref{l3.1}\text{--}\ref{l3.4}$, and \eqref{3.27} that
\begin{align}\label{3.55}
&\int_0^{\sigma(T)}\|F(\cdot,t)\|_{L^\infty}\mathrm{d}t\leq C\int_0^{\sigma(T)}\|F\|_{L^2}^{\frac13}\|\nabla F\|_{L^4}^{\frac23}\mathrm{d}t\notag\\
&\leq C\int_0^{\sigma(T)}\Big((2\mu+\lambda)^{\frac13}
\|\divv\mathbf{u}\|_{L^2}^{\frac13}+\|P-P(\tilde{\rho})\|_{L^2}^{\frac13}
+\|\mathbf{B}\|_{L^4}^{\frac23}\Big)
\Big(\|\sqrt{\rho}\dot{\mathbf{u}}\|_{L^4}^{\frac23}
+\|\mathbf{B}\cdot\nabla\mathbf{B}\|_{L^4}^{\frac23}\Big)
\mathrm{d}t\notag\\
&\leq C\sup_{0\leq t\leq T}\Big((2\mu+\lambda)^{\frac13}\|\divv\mathbf{u}\|_{L^2}^{\frac13}
+\|P-P(\tilde{\rho})\|_{L^2}^{\frac13}+\|\mathbf{B}\|_{L^4}^{\frac23}\Big)
\int_{0}^{\sigma(T)}\Big(
\|\dot{\mathbf{u}}\|_{L^{2}}^{\frac{1}{3}}\|\nabla\dot{\mathbf{u}}\|_{L^{2}}^{\frac{1}{3}}
+\|\mathbf{B}\|_{L^\infty}^{\frac23}\|\nabla\mathbf{B}\|_{L^4}^{\frac23}\Big)
\mathrm{d}t\notag\\
&\leq C(2\mu+\lambda)^{\frac{1}{6}}(2+M)^{\frac{1}{6}\exp\{3D_2(1+C_0)^6\}}
\int_{0}^{\sigma(T)}
\Big(\|\sqrt{\rho}\dot{\mathbf{u}}\|_{L^{2}}^{\frac{1}{3}}\|\nabla\dot{\mathbf{u}}
\|_{L^{2}}^{\frac{1}{3}}+C_0^\frac16\|\nabla\dot{\mathbf{u}}\|_{L^{2}}^{\frac{2}{3}}
+C_0^\frac{1}{12}\|\nabla\mathbf{B}\|_{L^2}^{\frac23}\|\nabla^2\mathbf{B}\|_{L^2}^{\frac12}
\Big)\mathrm{d}t\notag\\
&\leq(2\mu+\lambda)^{\frac{1}{6}}(2+M)^{\frac{1}{2}\exp\big\{3D_2(1+C_0)^6\big\}} \Bigg[\bigg(\int_{0}^{\sigma(T)}\|\sqrt{\rho}\dot{\mathbf{u}}\|_{L^{2}}^{2}\mathrm{d}t\bigg)^{\frac{1}{3}}\bigg(\int_{0}^{\sigma(T)}1\mathrm{d}t\bigg)^{\frac{2}{3}}
\notag\\&\quad+\bigg(\int_0^{\sigma(T)}t\|\nabla\dot{\mathbf{u}}\|_{L^2}^2
\mathrm{d}t\bigg)^{\frac13}\bigg(\int_0^{\sigma(T)}
t^{-\frac12}\mathrm{d}t\bigg)^{\frac23}+\sup_{0\leq t\leq T}\|\nabla\mathbf{B}\|_{L^2}^{2}
\bigg(\int_{0}^{\sigma(T)}\|\nabla^2\mathbf{B}\|_{L^{2}}^{2}\mathrm{d}t\bigg)^{\frac{1}{4}}\bigg(\int_{0}^{\sigma(T)}1\mathrm{d}t\bigg)^{\frac{3}{4}}
\Bigg]\notag\\
&\leq(2\mu+\lambda)^{\frac{1}{6}}\exp
\bigg\{(2+M)^{\exp\big\{\frac{17}{4}D_2(1+C_0)^6\big\}}\bigg\}.
\end{align}
Note that $\rho(t)$ takes values in $[\frac{3}{2}\hat{\rho},\frac{7}{4}\hat{\rho}]\subset [\hat{\rho},2\hat{\rho}]$ and $P(\rho)$ is increasing on $[0,\infty)$. Substituting $\eqref{3.55}$ into $\eqref{3.54}$ gives that
\begin{equation*}
    \ln\left(\frac{7}{4}\hat{\rho}\right)-\ln\left(\frac{3}{2}\hat{\rho}\right)
    \leq\frac{1}{\left(2\mu+\lambda\right)^{\frac{5}{6}}}
   \exp
\bigg\{(2+M)^{\exp\big\{\frac{9}{2}D_2(1+C_0)^6\big\}}\bigg\},
\end{equation*}
which is impossible if $\lambda$ satisfies \eqref{lam} with $D\geq 5D_2$. Therefore, there is no time $t_1$ such that $\rho(\mathbf{x}(t_1), t_1) = \frac{7}{4}\hat{\rho}$. Since $\mathbf{y}\in\mathbb{R}^2$ is arbitrary, it follows that $\rho<\frac{7}{4}\hat{\rho}$ \textit{a.e.} on $\mathbb{R}^2\times[0,T]$.

\textbf{Case 2:} $t_1>1$. We deduce from $\eqref{a1}_1$ and \eqref{1.6} that
\begin{equation*}
    \frac{\mathrm{d}}{\mathrm{d}t}(\rho(t)-\tilde{\rho})+\frac{1}{2\mu+\lambda}\rho(t)\Big(P(\rho(t))-P(\tilde{\rho})+\frac12|\mathbf{B}(\mathbf{x}(t),t)|^2\Big)=-\frac{1}{2\mu+\lambda}\rho(t)F(\mathbf{x}(t),t).
\end{equation*}
After multiplying the above equality by $|\rho(t)-\tilde{\rho}|(\rho(t)-\tilde{\rho})$, one has that
\begin{equation}\label{3.56}
    \frac{1}{3}\frac{\mathrm{d}}{\mathrm{d}t}\left|\rho(t)-\tilde{\rho}\right|^3+\frac{1}{2\mu+\lambda}\theta(t)\rho(t)|\rho(t)-\tilde{\rho}|^3
    =-\frac{1}{2(2\mu+\lambda)}\rho(t)(\rho(t)-\tilde{\rho})|\rho(t)-\tilde{\rho}|\big(2F+|\mathbf{B}|^2\big)(\mathbf{x}(t),t),
\end{equation}
where
\begin{equation*}
  \theta(t)\triangleq\frac{P(\rho(t))-P(\tilde{\rho})}{\rho(t)-\tilde{\rho}}.
\end{equation*}
If $\rho(t)$ takes values in $[\frac{3}{2}\hat{\rho},\frac{7}{4}\hat{\rho}]$,
applying the mean value theorem to $\theta(t)$ and then integrating \eqref{3.56} from $t_0$ to $t_1$, by similar arguments we obtain from Gagliardo--Nirenberg and Young's inequalities that
\begin{align*}
\hat{\rho}^{3}
&\leq\frac{C}{2\mu+\lambda}\int_{0}^{1}
\big(\|F(\cdot,t)\|_{L^\infty}+\|\mathbf{B}(\cdot,t)\|_{L^\infty}^2\big)\mathrm{d}t
+\frac{C}{2\mu+\lambda}\int_{1}^{T}
\big(\|F(\cdot,t)\|_{L^\infty}^3+\|\mathbf{B}(\cdot,t)\|_{L^\infty}^6\big)\mathrm{d}t\notag\\
&\leq(2\mu+\lambda)^{-\frac{5}{6}}\exp
\bigg\{(2+M)^{\exp\big\{\frac{17}{4}D_2(1+C_0)^6\big\}}\bigg\}
+\frac{C}{2\mu+\lambda}\int_{1}^{T}\|F\|_{L^{2}}\|\nabla F\|_{L^{4}}^{2}\mathrm{d}t\notag\\
&\leq(2\mu+\lambda)^{-\frac{5}{6}}\exp
\bigg\{(2+M)^{\exp\big\{\frac{9}{2}D_2(1+C_0)^6\big\}}\bigg\}
\notag\\&\quad
+(2\mu+\lambda)^{-\frac{1}{2}}\exp
\bigg\{(2+M)^{\frac{1}{2}\exp\big\{3D_2(1+C_0)^6\big\}}\bigg\}
\int_{0}^{T}
\Big(\|\sqrt{\rho}\dot{\mathbf{u}}\|_{L^{2}}\|\nabla\dot{\mathbf{u}}\|_{L^{2}}
+\|\nabla\mathbf{B}\|_{L^2}^{2}
(1+\|\nabla^2\mathbf{B}\|_{L^2}^2)\Big)\mathrm{d}t\notag\\
&\leq\frac{1}{(2\mu+\lambda)^{\frac12}}
\exp
\bigg\{(2+M)^{\exp\big\{\frac{19}{4}D_2(1+C_0)^6\big\}}\bigg\},
\end{align*}
which is impossible if $\lambda$ satisfies \eqref{lam} with $D\geq 5D_2$. Hence we conclude that there is no time $t_1$ such that $\rho(\mathbf{x}(t_1), t_1) = \frac{7}{4}\hat{\rho}$. Since $\mathbf{y}\in\mathbb{R}^2$ is arbitrary, it follows that $\rho<\frac{7}{4}\hat{\rho}$ \textit{a.e.} on $\mathbb{R}^2\times[0,T]$.
\end{proof}

Now we are ready to prove Proposition $\ref{p3.1}$.
\begin{proof}[Proof of Proposition \ref{p3.1}.]
Proposition \ref{p3.1} follows from Lemmas $\ref{l3.2}$--$\ref{l3.5}$ provided that $\lambda$ satisfies \eqref{lam} with $D\geq 5D_2$.
\end{proof}

\section{Proof of Theorem \ref{t1.1}}\label{sec4}
This section employs the \textit{a priori} estimates derived in Section $\ref{sec3}$ to finalize the proof of Theorem \ref{t1.1}.
\begin{proof}[Proof of Theorem \ref{t1.1}.]
Let $(\rho_0, \mathbf{u}_0, \mathbf{B}_0)$ be the initial data as described in Theorem \ref{t1.1}.
For $\epsilon>0$, let $j_\epsilon=j_\epsilon(\mathbf{x})$ be the standard mollifier,
and define the approximate initial data $(\rho_0^\epsilon, \mathbf{u}_0^\epsilon, \mathbf{B}_0^\epsilon)$:
\begin{align*}
\rho_0^\epsilon=J_\epsilon*\rho_0+\epsilon\triangleq[\rho_0]_\epsilon+\epsilon,~~
\mathbf{u}_0^\epsilon=J_\epsilon*\mathbf{u}_0\triangleq[\mathbf{u}_0]_\epsilon,~~
\mathbf{B}_0^\epsilon=J_\epsilon*\mathbf{B}_0\triangleq[\mathbf{B}_0]_\epsilon.
\end{align*}
Then we have
\begin{align*}
(\rho_0^\epsilon-\tilde{\rho},\mathbf{u}_0^\epsilon,\mathbf{B}_0^\epsilon)\in H^2\ \
\text{and} \ \ \inf_{\mathbf{x}\in\mathbb{R}^2}\{\rho_0^\epsilon(\mathbf{x})\}\geq\epsilon.
\end{align*}
Proposition $\ref{p3.1}$ implies that, for $\epsilon$ being suitably small,
\begin{align*}
0\leq\rho^\epsilon(\mathbf{x},t)\leq\frac{7}{4}\hat{\rho}~\textit{a.e.}~\mathrm{on}~\mathbb{R}^2\times[0,T]
\end{align*}
provided that $\lambda$ satisfies \eqref{lam}. Thus, Lemma $\ref{l2.1}$ yields the global existence and uniqueness of strong solutions $(\rho^\epsilon,\mathbf{u}^\epsilon,\mathbf{B}^\epsilon)$ to \eqref{a1} and \eqref{a3} with the initial data $(\rho_0^\epsilon,\mathbf{u}_0^\epsilon,\mathbf{B}_0^\epsilon)$.

Fix $\mathbf{x}\in\mathbb{R}^2$ and let $B_R$ be a ball of radius $R$ centered at $\mathbf{x}$.
Let $(F^\epsilon,\omega^\epsilon)$ be the functions $(F,\omega)$ with $(\rho,{\bf u},{\bf B})$ replaced by $(\rho^\epsilon,{\bf u}^\epsilon,{\bf B}^\epsilon)$.
Then, for $t\geq\tau>0$, one gets from Lemmas $\ref{E0}$, $\ref{l3.1}\text{--}\ref{l3.4}$, and Sobolev's inequality that
\begin{align}\label{4.1}
\langle\mathbf{u}^\epsilon(\cdot,t)
\rangle^{\frac12}&\leq C(1+\|\nabla\mathbf{u}^\epsilon\|_{L^4})\notag\\
&\leq C\Big(\|\sqrt{\rho^\epsilon}\dot{\mathbf{u}}^\epsilon\|_{L^2}^{\frac{1}{2}}
+\|\mathbf{B}^\epsilon\cdot\nabla\mathbf{B}^\epsilon\|_{L^2}^{\frac{1}{2}}\Big)
\|\nabla\mathbf{u}^\epsilon\|_{L^2}^\frac{1}{2}
+\frac{C}{2\mu+\lambda}\|\mathbf{B}^\epsilon\|_{L^{8}}^2+\frac{C}{2\mu+\lambda}\|P(\rho^\epsilon)-P(\tilde\rho+\epsilon)\|_{L^4}\notag\\
&\quad+
\frac{C}{2\mu+\lambda}\Big(\|\sqrt{\rho^\epsilon}\dot{\mathbf{u}}^\epsilon\|_{L^2}^{\frac{1}{2}}
+\|\mathbf{B}^\epsilon\cdot\nabla\mathbf{B}^\epsilon\|_{L^2}^{\frac{1}{2}}\Big)
\Big(\|P(\rho^\epsilon)-P(\tilde\rho+\epsilon)\|_{L^2}^{\frac{1}{2}}
+\|\mathbf{B}^\epsilon\|_{L^4}\Big)+C
\leq C(\tau).
\end{align}
Note that
\begin{align*}
\bigg|\mathbf{u}^\epsilon(\mathbf{x},t)-\frac{1}{|B_R(\mathbf{x})|}
\int_{B_R(\mathbf{x})}\mathbf{u}^\epsilon(\mathbf{y},t)\mathrm{d}\mathbf{y}\bigg|
&=\bigg|\frac{1}{|B_R(\mathbf{x})|}
\int_{B_R(\mathbf{x})}[\mathbf{u}^\epsilon(\mathbf{x},t)
-\mathbf{u}^\epsilon(\mathbf{y},t)]\mathrm{d}\mathbf{y}\bigg|\notag\\
&\leq\frac{1}{|B_R(\mathbf{x})|}C(\tau)\int_{B_R(\mathbf{x})}
|\mathbf{x}-\mathbf{y}|^{\frac12}\mathrm{d}\mathbf{y}\leq C(\tau)R^{\frac12}.
\end{align*}
Then, for $0<\tau\leq t_1<t_2<\infty$, it follows that
\begin{align}
|\mathbf{u}^\epsilon(\mathbf{x},t_2)-\mathbf{u}^\epsilon(\mathbf{x},t_1)|
&\leq\frac{1}{|B_R(\mathbf{x})|}\int_{t_{1}}^{t_{2}}
\int_{B_R(\mathbf{x})}|\mathbf{u}_{t}^{\epsilon}(\mathbf{y},t)
|\mathrm{d}\mathbf{y}\mathrm{d}t+C(\tau)R^{\frac12}\notag\\
&\leq CR^{-1}|t_{2}-t_{1}|^{\frac{1}{2}}\bigg(\int_{t_{1}}^{t_{2}}
\int_{B_R(\mathbf{x})}\left|\mathbf{u}_{t}^{\epsilon}(\mathbf{y},t)\right|^{2}
\mathrm{d}\mathbf{y}\mathrm{d}t\bigg)^{\frac{1}{2}}+C(\tau)R^{\frac12}\notag\\
&\leq CR^{-1}|t_{2}-t_{1}|^{\frac{1}{2}}\bigg(\int_{t_{1}}^{t_{2}}
\int_{B_R(\mathbf{x})}\left(|\dot{\mathbf{u}}^{\epsilon}|^{2}
+|\mathbf{u}^{\epsilon}|^{2}|\nabla\mathbf{u}^{\epsilon}|^{2}\right)
\mathrm{d}\mathbf{y}\mathrm{d}t\bigg)^{\frac{1}{2}}
+C(\tau)R^{\frac12}\notag\\
&\leq C(\tau)R^{-1}|t_{2}-t_{1}|^{\frac{1}{2}}+C(\tau)R^{\frac12},\notag
\end{align}
due to
\begin{align}
\int_{t_1}^{t_2}\int|\mathbf{u}^\epsilon|^2
|\nabla\mathbf{u}^\epsilon|^2\mathrm{d}\mathbf{x}\mathrm{d}t
&\leq C\sup_{t_1\leq t\leq t_2}\|\mathbf{u}^\epsilon\|_{L^\infty}^2\int_{t_1}^{t_2}
\int|\nabla\mathbf{u}^\epsilon|^2\mathrm{d}\mathbf{x}\mathrm{d}t\notag\\
&\leq C\sup_{t_1\leq t\leq t_2}\|\mathbf{u}^\epsilon\|_{L^2}^{\frac{2}{3}}
\|\nabla\mathbf{u}^\epsilon\|_{L^4}^{\frac{4}{3}}
\int_{t_1}^{t_2}\int|\nabla\mathbf{u}^\epsilon|^2\mathrm{d}\mathbf{x}\mathrm{d}t\leq C(\tau).\notag
\end{align}
Choosing $R=|t_2-t_1|^{\frac13}$, one can see that
\begin{equation}\label{4.2}
|\mathbf{u}^\epsilon(\mathbf{x},t_2)-\mathbf{u}^\epsilon(\mathbf{x},t_1)|
\leq C(\tau)|t_{2}-t_{1}|^{\frac{1}{6}},~~\text{for}~~0<\tau\leq t_1<t_2<\infty.
\end{equation}
The same estimates in \eqref{4.1} and \eqref{4.2} also hold for the magnetic filed $\mathbf{B}^\epsilon$, which implies that $\{\mathbf{u}^\epsilon\}$ and $\{\mathbf{B}^\epsilon\}$ are uniformly H\"older continuous away from $t=0$.

Hence, for any fixed $\tau$ and $T$ with $0<\tau<T<\infty$, it follows from Ascoli--Arzel\`{a} theorem that there is a subsequence $\epsilon_k\rightarrow0$ satisfying
\begin{equation*}
  \mathbf{u}^{\epsilon_k}\rightarrow \mathbf{u}~~\mathrm{and}~~
\mathbf{B}^{\epsilon_k}\rightarrow \mathbf{B}
~~\mathrm{uniformly}~\mathrm{in}~ B_T\times[\tau,T].
\end{equation*}
By the diagonalization argument, we obtain a subsequence $\{\epsilon_{k'}\}$ such that
\begin{equation}\label{4.3}
\mathbf{u}^{\epsilon_{k'}}\rightarrow \mathbf{u}~~\mathrm{and}~~
\mathbf{B}^{\epsilon_{k'}}\rightarrow \mathbf{B}
~~\mathrm{uniformly}~\mathrm{on}~\mathrm{compact}~\mathrm{sets} ~\mathrm{in}~\mathbb{R}^2\times(0,\infty), ~~\text{as}~~\epsilon_{k'}\rightarrow0.
\end{equation}
Moreover, by the standard compactness arguments as in \cite{EF01,PL98}, we can extract a further subsequence $\epsilon_{k''}\rightarrow0$ satisfying
\begin{equation}\label{4.4}
  \rho^{\epsilon_{k''}}-\tilde{\rho}\rightarrow\rho-\tilde{\rho}~~\mathrm{strongly}~\mathrm{in}~L^p(\mathbb{R}^2),~~\mathrm{for}~
  \mathrm{any}~p\in[2,\infty)~\mathrm{and}~t\geq0.
\end{equation}
Therefore, taking the limit of $\epsilon_{k''}\rightarrow0$ in \eqref{4.3} and \eqref{4.4}, we conclude that the limit triplet $(\rho,\mathbf{u},\mathbf{B})$ is a weak solution to the problem \eqref{a1}--\eqref{a3} in the sense of Definition $\ref{d1.1}$ and satisfies \eqref{reg}.
\end{proof}

\section{Proof of Theorem \ref{t1.2}}\label{sec5}
This section is devoted to the incompressible limit of \eqref{a1}--\eqref{a3} as the bulk viscosity tends to infinity.

\begin{proof}[Proof of Theorem \ref{t1.2}.]
Let $\{(\rho^\lambda,{\bf u}^\lambda, {\bf B}^\lambda)\}$ be the family of solutions to the Cauchy problem \eqref{a1}--\eqref{a3} from Theorem \ref{t1.1}. Applying \eqref{reg} and performing a similar argument as in \eqref{4.3}, we can see that there is a subsequence $\{(\rho^{\lambda_{k}},{\bf u}^{\lambda_{k}},{\bf B}^{\lambda_{k}})\}$ satisfying
\begin{gather}
{\bf u}^{\lambda_{k}}\rightarrow {\bf v},~~{\bf B}^{\lambda_{k}}\rightarrow {\bf b}
~~\text{uniformly on compact sets  in}~\mathbb R^2\times(0,\infty),\notag\\
\rho^{\lambda_{k}}-\tilde\rho\rightarrow \varrho-\tilde\rho\ \    \text{weakly in}\ \ L^p(\mathbb R^2),\ \ \text{for any}\ p\in[2,\infty)\ \text{and}\ {t\ge 0},\label{5.1}\\
\rho^{\lambda_{k}}\rightarrow \varrho\ \  \text{weakly* in}\ \ L^\infty(\mathbb R^2),\ \ \text{for any}\ {t\ge 0},\notag\\
\divv{\bf u}^{\lambda_{k}}\rightarrow 0~~\text{strongly  in}~L^2(\mathbb R^2\times(0,\infty)).\notag
\end{gather}
Hence, we conclude that $\divv{\bf v}=0$ and $(\varrho,{\bf v},{\bf b})$ satisfies \eqref{1.16}--\eqref{1.18} for all $C^1$
test functions $(\phi,\boldsymbol\psi)$ with uniformly bounded support in ${\bf x}$ for $t\in[t_1,t_2]$ and $\divv\boldsymbol\psi=0$ on $\mathbb R^2\times[0,\infty)$. Moreover, $(\varrho,{\bf v},{\bf b})$ has the following properties:
\begin{align}\label{5.2}
 0\leq\varrho({\bf x},t)\leq 2 \hat\rho\ \ \text{a.e. on} \
\mathbb R^2\times[0,\infty),
\end{align}
\begin{align}\label{5.3}
&\sup\limits_{t\ge 0}\big(\|\varrho-\tilde\rho\|_{L^2}^2+\|\sqrt{\varrho}{\bf v}\|_{L^2}^2+\|\nabla{\bf v}\|_{L^2}^2+\sigma\|\nabla^2{\bf v}\|_{L^2}^2
+\|\mathbf{b}\|_{H^1}^2+\sigma\|\nabla^2 \mathbf{b}\|_{L^2}^2\big)\notag\\
&\quad+\int_0^\infty\big(\mu\|\nabla{\bf v}\|_{L^2}^2+\|\nabla^2{\bf v}\|_{L^2}^2+\nu\|\nabla{\bf b}\|_{L^2}^2+\|\nabla^2 \mathbf{b}\|_{L^2}^2
\big)\mathrm{d}\tau\le C(C_0,M).
\end{align}

It remains to show \eqref{1.13}, \eqref{1.14}, and \eqref{1.15}. From the mass equation $\eqref{a1}_1$, we have that
\begin{equation*}
\partial_t(\rho^{\epsilon,\lambda}-\tilde\rho)^2+{\bf u}^{\epsilon,\lambda}\cdot\nabla(\rho^{\epsilon,\lambda}-\tilde\rho)^2
+2\rho^{\epsilon,\lambda}(\rho^{\epsilon,\lambda}-\tilde\rho)\divv{\bf u}^{\epsilon,\lambda}=0.
\end{equation*}
Integrating the above equality over $\mathbb R^2\times(0,t)$, we obtain that
\begin{align*}
&\big|\|(\rho^{\epsilon,\lambda}-\tilde\rho)(\cdot,t)\|_{L^2}^2
-\|\rho_0^\epsilon-\tilde\rho\|_{L^2}^2\big|
\\ & =\left|-\int_0^t\int(\rho^{\epsilon,\lambda}-\tilde\rho)^2\divv{\bf u}^{\epsilon,\lambda}\mathrm{d}{\bf x}\mathrm{d}\tau+2\int_0^t\int\rho^{\epsilon,\lambda}(\rho^{\epsilon,\lambda}
-\tilde\rho)\divv{\bf u}^{\epsilon,\lambda}\mathrm{d}{\bf x}\mathrm{d}\tau\right|
\\ & \le C\bigg(\int_0^t\big\|\rho^{\epsilon,\lambda}-\tilde\rho\big\|_{L^4}^4
\mathrm{d}\tau\bigg)^\frac{1}{2}\bigg(\int_0^t\big\|\divv{\bf u}^{\epsilon,\lambda}\big\|_{L^2}^2\mathrm{d}\tau\bigg)^\frac{1}{2}
\\ & \quad +C\sup\limits_{t\ge 0}\big\|\rho^{\epsilon,\lambda}(\cdot,t)\big\|_{L^\infty} \bigg(\int_0^t\big\|\rho^{\epsilon,\lambda}-\tilde\rho\big\|_{L^2}^2
\mathrm{d}\tau\bigg)^\frac{1}{2}\bigg(\int_0^t\big\|\divv{\bf u}^{\epsilon,\lambda}\big\|_{L^2}^2\mathrm{d}\tau\bigg)^\frac{1}{2}
 \le C(t)\lambda^{-\frac{1}{2}}.
\end{align*}
This together with \eqref{4.4} gives that
\begin{align*}
\big|\|(\rho^{\lambda}-\tilde\rho)(\cdot,t)\|_{L^2}^2
-\|\rho_0-\tilde\rho\|_{L^2}^2\big|
=\lim\limits_{\epsilon_{k''}\rightarrow 0}
\big|\|(\rho^{\epsilon_{k''},\lambda}-\tilde\rho)(\cdot,t)\|_{L^2}^2
-\|\rho_0^{\epsilon_{k''}}-\tilde\rho\|_{L^2}^2\big|
\le C(t)\lambda^{-\frac{1}{2}},
\end{align*}
which yields that
\begin{equation}\label{5.4}
\lim\limits_{\lambda\rightarrow\infty}\|(\rho^\lambda-\tilde\rho)(\cdot,t)\|_{L^2}
=\|\rho_0-\tilde\rho\|_{L^2},\ \ \text{for any}\ t\ge 0.
\end{equation}
Performing similar arguments as those in \eqref{3.11}--\eqref{3.13}, one gets that
\begin{align*}
\|\mathbf{v}\|_{H^1} \leq C(1+C_0)^\frac{1}{2}(1+\|\nabla\mathbf{v}\|_{L^2}),
\end{align*}
which along with \eqref{5.3} implies that
\begin{align*}
\int_0^T\int|{\bf v}|^2\mathrm{d}{\bf x}\mathrm{d}t&\le C\int_0^T(1+C_0)\big(1+\|\nabla\mathbf{v}\|_{L^2}^2\big)\mathrm{d}t\le C(T),
\end{align*}
as the desired \eqref{1.15}.

Next, in view of \eqref{1.15} and \eqref{5.3}, using the mollifier $j_\epsilon$ as test functions in \eqref{1.16}, we see that
\begin{equation}\nonumber\partial_t[\varrho]_\epsilon+{\bf v}\cdot\nabla [\varrho]_\epsilon=\divv\left([\varrho]_\epsilon {\bf v}\right)-\divv[\rho{\bf v}]_\epsilon~~ \text{a.e. on}\ \mathbb R^2\times(0,\infty),
\end{equation}
and furthermore,
\begin{equation*}
\partial_t\left([\varrho]_\epsilon-\tilde\rho\right)^2
+{\bf v}\cdot\nabla \left([\varrho]_\epsilon-\tilde\rho\right)^2=2\left([\varrho]_\epsilon-\tilde\rho\right)\big(\divv\left([\varrho-\tilde\rho]_\epsilon {\bf v}\right)-\divv\left[(\varrho-\tilde\rho){\bf v}\right]_\epsilon\big)~~ \text{a.e. on}~  \mathbb R^2\times(0,\infty).
\end{equation*}
Integrating the above equality over $\mathbb R^2\times(0,t)$, we have
\begin{align}\label{5.5}
\big|\|([\varrho]_{\epsilon}-\tilde\rho)(\cdot,t)\|_{L^2}^2- \|[\rho_0]_{\epsilon}-\tilde\rho\|_{L^2}^2\big|
 & \le C\sup\limits_{t\ge 0}\|[\varrho]_\epsilon-\tilde\rho\|_{L^\infty}\int_0^t
 \|\divv\left([\varrho-\tilde\rho]_\epsilon {\bf v}\right)-\divv\left[(\varrho-\tilde\rho){\bf v}\right]_\epsilon\|_{L^1}\mathrm{d}\tau
\notag \\
& \le C\int_0^t\|\divv\left([\varrho-\tilde\rho]_\epsilon {\bf v}\right)-\divv\left[(\varrho-\tilde\rho){\bf v}\right]_\epsilon\|_{L^1}\mathrm{d}\tau.
\end{align}
By Lemma \ref{lcom}, it follows that
\begin{equation*}
\|\divv\left([\varrho-\tilde\rho]_\epsilon {\bf v}\right)-\divv\left[(\varrho-\tilde\rho){\bf v}\right]_\epsilon\|_{L^1}\le C\|\varrho-\tilde\rho\|_{L^2}\|{\bf v}\|_{H^1}\in L^1(0,T),\ \ \text{for any}\ T>0.
\end{equation*}
This together with Lebesgue's dominated convergence theorem and Lemma \ref{lcom} leads to
\begin{align}\label{5.6}
&\lim\limits_{\epsilon\rightarrow0}\int_0^t
\|\divv\left([\varrho-\tilde\rho]_\epsilon {\bf v}\right)-\divv\left[(\varrho-\tilde\rho){\bf v}\right]_\epsilon\|_{L^1}\mathrm{d}\tau \notag \\
& =\int_0^t\lim\limits_{\epsilon\rightarrow0}\|\divv\left([\varrho-\tilde\rho]_\epsilon {\bf v}\right)-\divv\left[(\varrho-\tilde\rho){\bf v}\right]_\epsilon\|_{L^1}\mathrm{d}\tau=0.
\end{align}
Substituting \eqref{5.6} into \eqref{5.5}, we deduce that
\begin{align}\label{5.7}
\big|\|(\varrho-\tilde\rho)(\cdot,t)\|_{L^2}^2-\|\rho_0-\tilde\rho\|_{L^2}^2\big|= \lim\limits_{\epsilon\rightarrow0} \big|\| ([\varrho]_{\epsilon}-\tilde\rho)(\cdot,t)\|_{L^2}^2- \|[\rho_0]_{\epsilon}-\tilde\rho\|_{L^2}^2\big|=0,
\end{align}
which yields \eqref{1.14}.

Finally, combining \eqref{5.4} and \eqref{5.7}, one has that
\begin{align*}
\lim\limits_{\lambda\rightarrow\infty}\|(\rho^\lambda-\tilde\rho)(\cdot,t)\|_{L^2}
=\|(\varrho-\tilde\rho)(\cdot,t)\|_{L^2},\ \ \text{for any}\ t\ge 0,
\end{align*}
which along with \eqref{5.1} implies \eqref{1.13}. Consequently, $(\varrho,{\bf v},{\bf b})$ is a global weak solution to the inhomogeneous incompressible magnetohydrodynamic equations \eqref{1.5} in the sense of Definition \ref{d1.2}.
\end{proof}

\section*{Conflict of interests}
The authors declare that they have no conflict of interests.

\section*{Data availability}
No data was used for the research described in the article.

\end{document}